\documentclass[12pt,reqno]{amsart}

\usepackage{amsmath,amssymb,amsthm,color,psfrag,fourier}
\usepackage{microtype}
\usepackage[a4paper, total={6.5in, 8.5in}]{geometry}

\usepackage{graphicx}

\usepackage[final,
pdftitle={TITLE},
pdfauthor={NAME},hypertexnames=false]{hyperref}
\usepackage{epstopdf}

\theoremstyle{plain}

\newtheorem*{rem*}{Remark}
\theoremstyle{definition}
\newtheoremstyle{case}{}{}{}{}{}{:}{ }{}
\newtheorem{case}{Case}

\numberwithin{subcase}{case}

\theoremstyle{plain}
\newtheorem{thm}{Theorem}[section]  %reset theorem numbering for each chapter
\newtheorem{propn}[thm]{Proposition}
\newtheorem{lemm}[thm]{Lemma}
\newtheorem{corll}[thm]{Corollary}
\newtheorem{remark}[thm]{Remark}

\theoremstyle{definition}
\newtheorem{defn}[thm]{Definition} %definition numbers are dependent on theorem numbers
\newtheorem{ex}[thm]{Example}

\def\scl[#1][#2]{{\scalebox{#1}{#2}}}

\def\oep{\overline{e}_{+}}
\def\oem{\overline{e}_{-}}

\def\ep{e_{+}}
\def\eem{e_{-}}

\def\epp{e_{++}}
\def\app{a_{++}}

\def\vpp{v_{++}}
\def\vmm{v_{--}}

\def\oap{\overline{a}_{+}}
\def\oam{\overline{a}_{-}}

\def\oapp{\overline{a}_{++}}

\def\ovpp{\overline{v}_{++}}
\def\ovmm{\overline{v}_{--}}

\def\ap{a_{+}}
\def\am{a_{-}}
\def\vp{v_{+}}
\def\vm{v_{-}}

\def\oapp{\overline{a}_{++}}

\def\ovpp{\overline{v}_{++}}
\def\ovmm{\overline{v}_{--}}

\newcommand{\OO}{\mathcal{O}}

\def\cob[#1]{\textcolor{blue}{#1}}
\def\cog[#1]{\textcolor{green}{#1}}
\def\cor[#1]{\textcolor{red}{#1}}

\newcommand{\param}%
	{{\mathchoice{\mkern1mu\mbox{\raise2.2pt\hbox{$\centerdot$}}\mkern1mu}%
	{\mkern1mu\mbox{\raise2.2pt\hbox{$\centerdot$}}\mkern1mu}%
	{\mkern1.5mu\centerdot\mkern1.5mu}{\mkern1.5mu\centerdot\mkern1.5mu}}}

\begin{document}

%%%%%%%%%%%%%%%%%%%%%%%%%%%%%%%%%%%%%%%%%%%%%%%%%%%%%%%%%%%%%%%%%%%%%%%%%%%%% 

\title[A Structure Theorem for Bad 3-Orbifolds]{A Structure Theorem for Bad 3-Orbifolds}

%\title[The Classification of Bad 3-Orbifolds]{The Classification of Bad 3-Orbifolds}

%\title[A Prime Decomposition for Bad 3-Orbifolds]{A Prime Decomposition for Bad 3-Orbifolds}

%\title[A Structure Theorem for Bad 3-Orbifolds]{A Structure Theorem for Bad 3-Orbifolds}

%\title[Obtaining a Good 3-Orbifold From a Bad 3-Orbifold]{Obtaining a Good 3-Orbifold From a Bad 3-Orbifold}

\author[R. Lehman]{R. Lehman}
\author[Yo'av Rieck]{Yo'av Rieck}
\address{Dept.\ of Mathematics \\
  University of Arkansas\\
  Fayetteville, AR 72701}
\email{\href{mailto:Rachellehman0530@gmail.com}{Rachellehman0530@gmail.com}}
\email{\href{mailto:yoav@uark.edu}{yoav@uark.edu}}

%%%%%%%%%%%%%%%%%%%%%%%%%%%%%%%%%%%%%%%%%%%%%%%%%%%%%%%%%%%%%%%
\newif\ifDraft
\Draftfalse
%\Drafttrue
%%%%%%%%%%%%%%%%%%%%%%%%%%%%%%%%%%%%%%%%%%%%%%%%%%%%%%%%%%%%%%%

\begin{abstract}
We \em explicitly \em construct a collection of bad 3-orbifolds, \(\mathcal{X}\), satisfying the following properties:
	\begin{enumerate}
	\item The underlying topological space of any \(X \in \mathcal{X}\) is homeomorphic to $S^2\times I$ or $(S^2\times S^1)\backslash B^3$.
	\item The boundary of any \(X \in \mathcal{X}\) consists of one or two spherical 2-orbifolds.
	\item Any bad 3-orbifold is obtained from a good 3-orbifold by repeating, finitely many times, the following operation: remove one or two orbifold-balls, and glue in some \(X \in \mathcal{X}\). 
	\end{enumerate}
 Conversely, any bad 3-orbifold \(\OO\) contains some \(X \in \mathcal{X}\) as a sub-orbifold; we call removing \(X\) and capping the resulting boundary \em cut-and-cap.\em\ Then by cutting-and-capping finitely many times we obtain a good orbifold.

\ifDraft
\textcolor{orange}{{\tiny
having the following property: given any bad 3-orbifold, $\OO$, it admits an embedded suborbifold $X\in\mathcal{X}$ such that after removing this member from $\OO$, and capping the resulting boundary, and then iterating this process finitely many times, you obtain a good 3-orbifold. 
Reversing this process gives us a procedure to obtain any possible bad 3-orbifold starting with a good 3-orbifold. 
Each member of $\mathcal{X}$ has 1 or 2 spherical boundary components and has underlying topological space $S^2\times I$ or $(S^2\times S^1)\backslash B^3$. }}
\fi
\end{abstract}

\clearpage\maketitle
\thispagestyle{empty}

%%%%%%%%%%%%%%%%%%%%%%%%%%%%%%%%%%%%%%%%%%%%%%%%%%%%%%%%%%%%%%%%%%%%%%%%%%%%% 

%Tears, idle tears, I know not what they mean, \\
%Tears from the depth of some divine despair \\
%Rise in the heart, and gather to the eyes, \\
%In looking on the happy autumn-fields, \\
%And thinking of the days that are no more. \\[8pt]
%--- Lord Tennyson
  
\begin{flushright}
C'est tellement myst\'erieux, le pays des larmes.\\[4pt]
{\scriptsize --- Antoine De Saint-Exupery}
\end{flushright}

\bigskip

\section{Introduction}\label{sec:introduction}

Sometime around 1980 WP Thurston stated and sketched the proof of his celebrated \em Orbifold Theorem\em, showing that if a 3-dimensional orbifold \(\OO\) with non-empty singular set is \em good \em --- covered by a manifold --- then \(\OO\) admits a decomposition into geometric pieces (for definitions see Section~\ref{sec:Preliminaries}, and  for a discussion and proof see, for example,~\cite{thurston}, \cite{boileau}, \cite{MR2060653}, \cite{cooper}, \cite{MR1844891}). 
The case of an empty singular set --- manifolds --- is Thurston's Geometrization Conjecture,  proved by Perelman~\cite{perelman2}
~\cite{perelman1}~\cite{perelman3} (a proof of the Orbifold Theorem along these lines can be found in ~\cite{MR3244330}).  These theorems have been immensely influential and allowed the  study of various aspects of 3-dimensional topology by considering well-understood geometric pieces and combining them.  Discussion about 3-orbifolds can also be found in~\cite{ConwayBurgielGoodmanStrauss} 
and~\cite{Montesinos}.
In~\cite{BonahonSiebenmann} the reader can find the classification of Seifert fibered 3-orbifolds.

By its very definition a \em bad orbifold \em cannot be constructed from geometric pieces, and so the assumption that the orbifold is good cannot be removed. Our goal in this paper is to present an explicit structure theorem for bad orbifolds that allows their study by means of induction.

We now provide some more details about our strategy and results. Let $\OO$ be a closed, orientable (possibly bad) 3-orbifold. Thurston's Orbifold Theorem and Perelman's Geometrization Theorem state that after performing prime decomposition (decomposing $\OO$ along spherical 2-orbifolds --- 2-orbifolds covered by \(S^{2}\)) 
and then JSJ decomposition (decomposing along certain toroidal 2-orbifolds --- 2-orbifolds covered by the 2-torus) 
each of the pieces obtained, say \(\OO'\), satisfies exactly one of the following:
\begin{enumerate}
\item \(\OO'\) is \em geometric \em (the geometric quotient of one of Thurston's eight geometries), or
\item \(\OO'\) contains an embedded bad 2-orbifold.
\end{enumerate}
Consequently \(\OO\) itself is bad if and only if it contains an embedded bad 2-orbifold, which we can take to be either a teardrop or a bad-football (see Proposition~\ref{torf}). 
Using this we show that if $\OO$ is bad, then it contains a compact 3-suborbifold, $X$, satisfying the following properties: 
\begin{enumerate}
\item $X$ is a member of one of 
twelve explicitly constructed families, 
\ifDraft
\textcolor{orange}{\tiny {\bf draft only:} I split Forms 4 and 5 to 4.a, 4.b, 5.a, and 5.b} 
\fi
\item $X$ contains an embedded teardrop or bad-football, 
\item $\partial X$ consists of one or two spherical 2-orbifolds, and 
\item $X$ has underlying topological space $S^2\times I$ or $(S^2\times S^1)\backslash B^3$.
\end{enumerate}
\noindent Let the process of removing $X$ from $\OO$ and capping the resulting boundary components by attaching a cone be called \textit{cutting and capping $X$}.

The main result of this paper is: 

\begin{thm}
Let $\OO$ be a closed, orientable 3-orbifold. Then there exists orbifolds $\OO_0,\OO_1,\dots,\OO_{n+m}$ (for some $n,m\geq0$) so that $\OO_i$ is obtained from $\OO_{i-1}$ by cutting and capping $X_i$, where:
\begin{enumerate}
\item $\OO=\OO_0$,
\item $(1\leq i\leq n)$ $X_i$ is listed in Theorem~\ref{localteardrop},
\item $(n+1\leq i\leq n+m)$ $X_i$ is listed in Theorem~\ref{localfootball},
\item $\OO_{n+m}$ is good.
\end{enumerate}
In particular, the boundary of each $X_i$ consists of 1 or 2 spherical 2-orbifolds, and its underlying topological space is $S^2\times I$ or $(S^2\times S^1)\backslash B^3$.
\label{mainthm}
\end{thm} 

Reversing the process of Theorem \ref{mainthm} gives us the following corollary.

\begin{corll}
\label{cor:ConstructingBadOrbifolds}
Any closed, orientable 3-orbifold is obtained from some good 3-orbifold by a repeating the following process finitely many times:
	\begin{enumerate}
	\item  Removing 1 or 2 orbifold 3-balls, and ---
	\item Attaching a bad 3-orbifold listed in Theorem~\ref{localfootball} or~\ref{localteardrop}. 
	\end{enumerate}
\end{corll}

\begin{remark}
{\rm A few comments are in order:
\begin{enumerate}

\item The constructions in Theorems~\ref{localfootball} and~\ref{localteardrop} are explicit and therefore Corollary~\ref{cor:ConstructingBadOrbifolds} lends itself to the constructions of a variety of examples.  We invite the reader to do so.

\item The reader is also invited to simplify our constructions.  While doing so, keep in mind that the constructions are designed so that the boundary of \(X\) is spherical.  This is essential: it allows \em capping the boundary.\em\ This is, of course, similar to prime decomposition of \(3\)-manifolds.

 \item Since $X$ may have 2 boundary components, it is possible that \(\OO_{n+m}\) is disconnected even if \(\OO\) is connected.
\end{enumerate}
}
\end{remark}

\noindent
\textbf{Outline.} The content of each section is summarized here:
\begin{description}
\item[Section \ref{sec:Preliminaries}] Review of standard definitions, mostly to fix ideas. 
\item[Section~\ref{sec:Teardrops are Bounded}] A bound on the number of disjoint, pairwise non-parallel teardrops in $\OO$ (the analogous statement for bad-footballs is false, see Example~\ref{ExampleOfInfinitelyManyTears}).
\item[Section~\ref{sec:Bad-Footballs are Bounded}] A bound on the number of disjoint, pairwise non-parallel, bad-footballs in $\OO$ in the absence of teardrops. 
\item[Sections \ref{sec:Local Pictures about a Bad-Football} and \ref{sec:Local Pictures about a Teardrop}] Classification of compact 3-suborbifolds needed for Theorem \ref{mainthm}. 
\item[Section \ref{sec:Proof of Main Theorem}] Proof of Theorem \ref{mainthm}.
\end{description}

\noindent
\textbf{Acknowledgment.} This project grew out of discussions with Michael McQuillan and the authors are grateful for his insight and suggestions. The authors also thank Daryl Cooper, Chaim Goodman--Strauss, and Neil Hoffman for helpful correspondence and conversations.  We are grateful to the referees for helpful suggestions. Yo'av Rieck was partially supported by the Simons Foundation through grant \#637880.

%%%%%%%%%%%%%%%%%%%%%%%%%
%%%%%%%%%%%%%%%%%%%%%%%%%
%%%%%%%%%%%%%%%%%%%%%%%%%

\section{Preliminaries}\label{sec:Preliminaries}

\noindent
We assume familiarity with the basic notions of 3-manifold topology (see, for example,~\cite{hempel},~\cite{jaco} and ~\cite{schultens}), including the basic definitions related to orbifolds (see, for example, \cite{cooper} or \cite{boileau}). To fix ideas we summarize the main definitions here.  
Most of the notation we use is standard, with \(N\), \(E\), and \(\partial\) denoting closed normal neighborhood, exterior, and boundary. 
We use \(| \OO|\) to denote the underlying topological space of an orbifold \(\OO\), and we use \(\Sigma\) to denote its singular set.  An \(n\)-orbifold is called \em orientable \em if its underlying topological space is orientable and locally it is given as the quotient of a finite group whose action on \(\mathbb{R}^{n}\) is orientation preserving.  An orbifold is \em smooth \em if its singular set is empty; \(B^{3}\) is a smooth open \(3\)-ball. 
\ifDraft
When we use the notation \(B^{3}\) for a subspace of a 3-orbifold, we mean the interior of a closed embedded \(3\)-ball which is disjoint from \(\Sigma\).\fi

\medskip

\noindent{\bf \(2\)-orbifolds.}
Any 2-orbifold that we consider is assumed to be closed and orientable (except in the proof of Proposition~\ref{torf} where we are forced to consider non-orientable 2-orbifolds).
Any point \(p\) in such a 2-orbifold has a \em standard neighborhood \em  that is either a smooth disk or a disk with a single singular point, corresponding to \(p\), of weight \(a>1\) (by which we mean that it is the quotient of a disk in \(\mathbb{R}^{2}\) under the group generated by a rotation by angle \(2\pi/a\), where \(a \geq 2\) is an integer).   A teardrop, $S^2(a)$, is a 2-orbifold with $|S^2(a)| \cong S^2$ (the smooth 2-sphere) and exactly one singular point of weight $a>1$. A bad-football, $S^2(a,b)$, is a \(2\)-orbifold with $|S^2(a,b)| \cong S^2$ and exactly two singular points of unequal weights $a$ and $b$, where  we assume $b>a>1$.  It is well-known that teardrops and bad-footballs are closed, orientable, bad 2-orbifolds.  

A \textit{spherical 2-orbifold} is a 2-orbifold that is covered by $S^2$ (not to be confused with a 2-orbifold whose underlying topological space is \(S^{2}\)).  
A closed spherical orbifold of the from \(S^{2}(a_{1},a_{2},a_{3})\) must satisfy:
\[
\frac{1}{a_{1}} + \frac{1}{a_{2}} + \frac{1}{a_{3}} > 1
\] 
This inequality allows us to list all the spherical \(2\)-orbifolds:\label{2orbifolds}
	\begin{itemize}
	\item \(S^{2}\)
	\item \(S^{2}(a,a)\) for any \(a \geq 2\)
	\item \(S^{2}(2,2,a)\) for any \(a \geq 2\)
%	\item \(S^{2}(2,3,a)\) for \(a = 3,4\) or \(5\).
	\item \(S^{2}(2,3,3)\),  \(S^{2}(2,3,4)\), and \(S^{2}(2,3,5)\)
	\end{itemize}

A \textit{toroidal 2-orbifold} is a 2-orbifold that is covered by the 2-torus  (for example, \(S^{2}(3,3,3)\) is toroidal and is not spherical).  

\medskip

\noindent{\bf \(3\)-orbifolds.} 
By an \em orbifold \em we mean a compact, orientable \(3\)-orbifold.
Any point \(p\) in the interior of an orbifold has a \em standard neighborhood \em
\ifDraft
\textcolor{orange}{\tiny{\bf Draft only:}
 which is an \em orbifold ball\em, that is, a neighborhood whose universal cover is a smooth closed 3-ball.  \bf Orbifold ball \(\Leftrightarrow\) cone on spherical boundary?}
 \fi
which is a cone on a closed spherical 2-orbifold; hence its underlying topological space is a  closed 3-ball and (by the classification of finite group actions on \(S^{2}\)) its singular set is one of the following:
	\begin{itemize}
	\item Empty.
	\item A single unknotted arc, containing \(p\), of weight \(a>1\).
	\item Exactly three arcs, each connecting the boundary to \(p\), so that the union of any two of these arcs is an unknotted arc.  
	The weights of the arcs (say \(a_{1}\), \(a_{2}\), and \(a_{3}>1\)) satisfy 
	\[
	\frac{1}{a_{1}} + \frac{1}{a_{2}} + \frac{1}{a_{3}} > 1
	\] 
	\end{itemize}	
We observe that each component of the singular set of a closed orbifold has one of the following forms:
	\begin{itemize}
	\item A smooth simple closed curve (here ``smooth'' means ``contains no vertex'').
	\item A trivalent graph, and in particular, the graph has no leaves (valence one vertices).
	\end{itemize}

To clarify our usage of the term edge we define it here:

\begin{defn}[edge]
\label{Definition:TypesOfEdges}
{\rm
By an \em edge \em of \(\Sigma\) we mean either a smooth simple closed curve, or an edge ending in two vertices, which may or may not be distinct.}
\end{defn}

\medskip\noindent
The following proposition is an application of Thurston's Orbifold Theorem that allows us to avoid disks with corner lines and corner reflectors:\footnote{We refer the reader to~\cite[page~408]{scott} for the definition of corner lines and corner reflectors; in light of Proposition~\ref{torf}, they will not be considered in this paper beyond the proof of the  proposition.}   

\begin{propn}\label{torf}
A closed, orientable, bad 3-orbifold admits an embedded teardrop or bad-football.
\end{propn}

\begin{proof}
Let $\OO$ be a  closed orientable bad 3-orbifold.
Thurston's Orbifold Theorem implies that $\OO$ admits an embedded, closed, connected, bad 2-suborbifold, say $S$. 

It follows from~\cite[Theorem~2.3, page~425]{scott} that \(S\) is one of the following:\footnote{The notation used in~\cite{scott} is $D^2(a)$ and $D^2(a,b)$, but we will use the notation $\widehat D^2(a)$ and $\widehat D^2(a,b)$ instead to avoid conflicts with the notation used below.}
\begin{enumerate}
\item $S=S^2(a)$ with $a>1$,
\item $S=S^2(a,b)$ with $b>a>1$,%\footnote{In~\cite{scott} the condition is \(a \neq b\), but if one of the two equals \(1\) we in fact have a tear drop; similarly for \(\widehat D^2(a,b)\) below.}
\item $S=\widehat D^2(a)$, where the weight of the corner reflector is $a>1$, or
\item $S=\widehat D^2(a,b)$, where the weights of the corner reflectors are $b>a>1$.
\end{enumerate}
If $S=S^2(a)$ or $S^2(a,b)$ we are done.
Suppose that $S=\widehat D^2(a)$ or $\widehat D^2(a,b)$, see Figure~\ref{Figure for Proposition torf} where \(S\) is shaded.
\(S\) admits a neighborhood \(U\) with \(|U| \cong B^{3}\), and since the corner reflector points of $S$ have valence three, \(\Sigma \cap \partial U\) is either in one point of weight \(a>1\) (corresponding to case~3) or two points of weights \(b>a>1\) (corresponding to case~4).  Thus \(\partial U\) is either a teardrop or a bad-football.
\begin{figure}[h!]
\psfrag{2}{\scl[.7][$2$]}
\psfrag{p}{\scl[.7][$a$]}
\psfrag{q}{\scl[.7][$b$]}
\centerline{\includegraphics[scale=0.25]{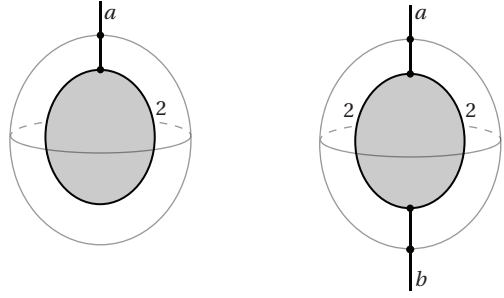}}
\caption{Figure for Proposition \ref{torf}}
\label{Figure for Proposition torf}
\end{figure}
\end{proof}

In the next three lemmas we study the local picture near \(\Sigma\),  the singular set of a closed orbifold.  Note that every component \(V\) of \(|N(\Sigma)|\) is a handlebody with the \(|\Sigma| \cap V\) its core.  In particular, \(V\) is irreducible.

\begin{lemm}
If $B \subset N(\Sigma)$, where $|B|$ is a closed 3-ball, then every component of  $\Sigma\cap B$ is a tree and all its leaves are on the boundary.
\label{leavesonboundary}
%If $\overline{B} \subset N(\Sigma)$, where $|\overline{B}|$ is a closed 3-ball, then every component of  $\Sigma\cap \overline{B}$ is a tree and all its leaves are on the boundary.
%\label{leavesonboundary}
\end{lemm}

\begin{proof}
First assume, for a contradiction, that $\Sigma\cap B$ contains a cycle, say \(c\).  Then $\pi_1(|c|)\not\hookrightarrow \pi_1(|B|)$ as $|B|$ is simply connected, which is a contradiction as \(|c|\) represents a nontrivial element of \(\pi_{1}(|\Sigma|)\) and $\pi_1(|\Sigma|)$ injects into $\pi_1(|N(\Sigma)|)$.  This shows that every component of $\Sigma\cap B$ is a tree. 

No leaf of $\Sigma\cap B$ can be in the interior of \(B\) since \(\Sigma\) has no vertex of valence one.
\end{proof}

Consequently:

\begin{lemm}
\(N(\Sigma)\) admits no bad orbifolds.
\label{Lem:NoBadOrbsInNS}
\end{lemm}

\begin{proof}
Let \(S \subset N(\Sigma)\) be an embedded 2-orbifold with \(|S|\) a 2-sphere and \(S \cap \Sigma\) one or two points.  Any bad 2-orbifold in \(N(\Sigma)\) certainly satisfies these conditions.  Irreducibility of \(|N(\Sigma)|\) implies that \(|S|\) bounds a closed ball \(|B| \subset |N(\Sigma)|\).  By Lemma~\ref{leavesonboundary}, \(\Sigma \cap B\) consists of trees with their leaves on the boundary.  Since every tree has at least two leaves, we have that  \(S \cap \Sigma\) consists of exactly two points.  The only tree with exactly two leaves is a single arc, implying that the two singular points on \(S\) have the same weight, so \(S\) is a good football.
\end{proof}

And for footballs in \(V\) we have:

\begin{lemm}
Any football \(F \subset N(\Sigma)\) bounds an orbifold isomorphic to \(D^{2}(a) \times I\) (where \(D^{2}\) is a 2-disk and \(a>1\)).
\label{Lem:GoodFootballsInNS}
\end{lemm}

\begin{proof}
\(F\) is a good football by Lemma~\ref{Lem:NoBadOrbsInNS}; in the proof of that lemma we observed that \(F\) bounds \(B\), where \(|B|\) is a closed ball and \(\Sigma \cap B\) is a single arc.  Since the core of a handlebody is locally unknotted, we have that this arc is unknotted which gives the desired product structure.
\end{proof}

We conclude this section with the following simple lemma:
\begin{lemm}
Let $a_1^*,a_2,$ and $a_3$ be the weights of three edges of \(\Sigma\) that meet at a vertex of \(\Sigma\).
If \(a_{1}^{*} \geq a_{1}\) then $S^2(a_1,a_2,a_3)$ is a spherical 2-orbifold.
\label{firstcut}
\end{lemm}

\begin{proof}
The weights of edges meeting at a vertex satisfy $\frac{1}{a_1^*}+\frac{1}{a_2}+\frac{1}{a_3}>1$. By assumption $a_1^{*} \geq a_1$, and thus $\frac{1}{a_1}+\frac{1}{a_2}+\frac{1}{a_3} \geq \frac{1}{a_1^*}+\frac{1}{a_2}+\frac{1}{a_3} >1$. Hence  $S^2(a_1,a_2,a_3)$ is a spherical 2-orbifold.
\end{proof}

%%%%%%%%%%%%%%%%%%%%%%%%%%%%%%%%%%%%%%%%%%%%%%%%%%%%%%%%%%%%%%%%%%%%%%%%%%%%%

\section{A Bound on the Number of Teardrops}\label{sec:Teardrops are Bounded}
By Proposition~\ref{torf} we have that an orbifold $\OO$ is bad if and only if it contains an embedded teardrop or bad-football. As our goal is to remove all bad components from \(\OO\) we must first know that there exists a bound on their number, starting with teardrops:

\begin{thm}[finiteness of teardrops]
Given a closed, orientable 3-orbifold, there exists a bound on the number of disjointly embedded, pairwise non-parallel teardrops. 
\label{badbounded}
\end{thm}

For a teardrop \(T\), we consider \(T \cap N(\Sigma)\) and \(T \cap E(\Sigma)\)  (recall that \(E(\Sigma)\) is the exterior of the singular set). We first prepare two lemmas:

\begin{lemm}
Let \(T \subset \OO\) be a teardrop.  Then $T\cap E(\Sigma)$ is an essential disk in $E(\Sigma)$.
\label{prop1}
\end{lemm}

\begin{proof}
Obviously, \(T \cap N(\Sigma)\) is a disk that intersects \(\Sigma\) in a single point. 
Since \(|N(\Sigma)|\) consists of handlebodies for which \(|\Sigma|\) are cores, we have that 
\(T\cap \partial E(\Sigma)\) is essential in \(\partial E(\Sigma)\); the lemma follows.
\end{proof}

\begin{lemm}
Let $T_1$ and $T_2$ be teardrops disjointly embedded in $\OO$.
If $T_1 \cap E(\Sigma)$ and $T_2  \cap E(\Sigma)$ are parallel in $E(\Sigma)$, then $T_1$ and $T_2$ are parallel in $\OO$.
\label{teardropparallel}
\end{lemm}

\begin{proof}
Suppose that $T_1 \cap E(\Sigma)$ and $T_2  \cap E(\Sigma)$ are parallel in $E(\Sigma)$.  Then, in particular, $\partial (T_1 \cap E(\Sigma))$ and $\partial (T_2 \cap E(\Sigma))$ co-bound an annulus \(A \subset \partial N(\Sigma)\).  Then \(\left(T_1 \cap N(\Sigma)\right) \cup A \cup \left(T_2 \cap N(\Sigma)\right)\) is a football in \(N(\Sigma)\) and bounds a product region by Lemma~\ref{Lem:GoodFootballsInNS} which, together with the product region between $T_1 \cap E(\Sigma)$ and $T_2  \cap E(\Sigma)$ in $E(\Sigma)$, shows that \(T_{1}\) and \(T_{2}\) are parallel.
\end{proof}

\begin{proof}[Proof of Theorem~~\ref{badbounded}]
By Lemmas~\ref{prop1} and~\ref{teardropparallel} the intersection of a family of disjoint, pairwise non-parallel teardrops with \(E(\Sigma)\) forms a family of disjoint, pairwise non-parallel, essential disks in \(E(\Sigma)\), whose number is bounded by Kneser--Haken finiteness. 
\end{proof}

%%%%%%%%%%%%%%%%%%%%%%%%%%%%%%%%%%%%
%%%%%%%%%%%%%%%%%%%%%%%%%%%%%%

\section{A Bound on the Number of Bad-Footballs}\label{sec:Bad-Footballs are Bounded}

Having bounded the number of teardrops we now turn our attention to bad-footballs; to that end we place a restriction on teardrops:

\begin{thm}[Finiteness of bad-footballs]
\label{thm:Finiteness_of_bad_footballs}
Let \(\OO\) be a closed orientable orbifold, and suppose that no component of \(\OO\) admits disjointly embedded teardrops 
of distinct weights. 
%\(S^{2}(a)\) and \(S^{2}(b)\) with \(a \neq b\).  
Then there is a bound on the number of disjointly embedded, pairwise non-isotopic, bad-footballs. 
\end{thm}

\begin{ex}
\label{ExampleOfInfinitelyManyTears} 
Let \(\OO\) be a connected 3-orbifold that admits two disjointly embedded teardrops, $S^2(a)$ and $S^2(b)$, with \(a \neq b\).  Tubing the two together in \(E(\Sigma)\) we obtain a bad-football that compresses down to $S^2(a)$ and $S^2(b)$ (see Figure~\ref{compression}, where \(D\) is a compressing disk for the tube).  We can tube $S^2(a)$ and $S^2(b)$ again, using a knotted tube nested inside the first tube.  Repeating this process (see Figure~\ref{infinitelyfb}) we obtain arbitrarily many disjointly embedded, pairwise non-isotopic bad-footballs.  Thus the assumptions of Theorem~\ref{thm:Finiteness_of_bad_footballs} cannot be removed.  

\begin{figure}[h!]
\psfrag{s}{\scl[.7][$\Sigma$]}
\psfrag{a}{\scl[.7][$a$]}
\psfrag{b}{\scl[.7][$b$]}
\psfrag{D}{\scl[.7][$D$]}
\centerline{\includegraphics[scale=0.25]{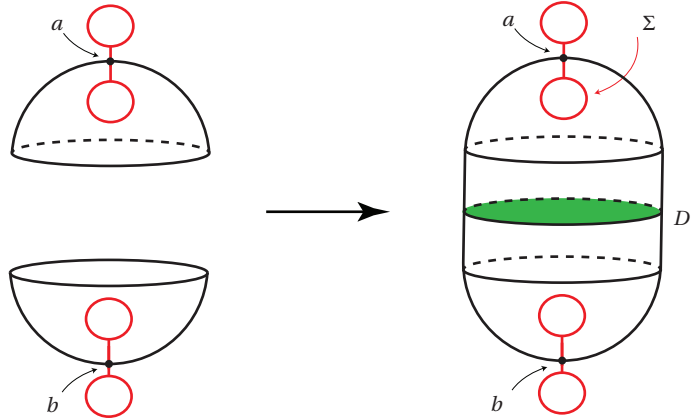}}
\caption{The bad-football containing a compression disc.}
\label{compression}
\end{figure}

\begin{figure}[h!]
\psfrag{s}{\scl[.7][$\Sigma$]}
\psfrag{D}{\scl[.7][$D$]}
\hspace{-6cm}
\centerline{\includegraphics[scale=0.5]{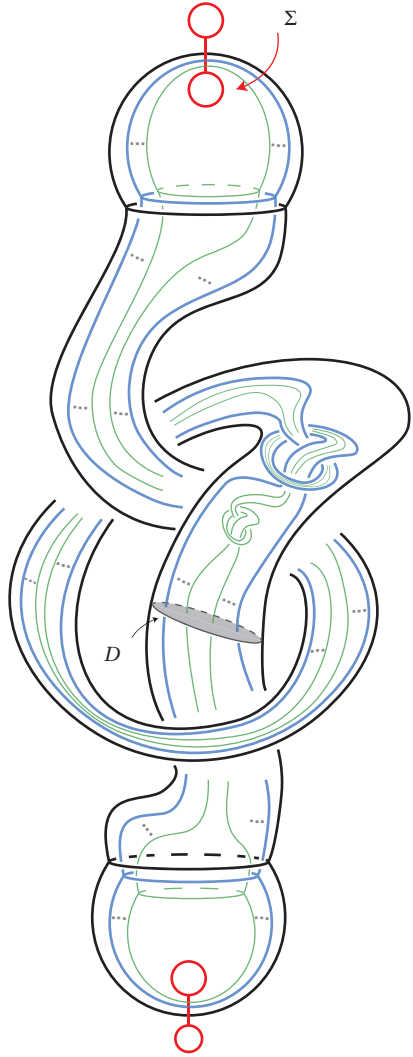}}
\caption{We may obtain an unbounded sequence of non-parallel, compressible bad-footballs.}
\label{infinitelyfb}
\end{figure}
\end{ex}

In contrast to the case of teardrops (Lemma~\ref{prop1}), a bad-football may not give rise to an essential annulus in \(E(\Sigma)\) (as it may boundary compress). We prepare three lemmas for the proof of Theorem~\ref{thm:Finiteness_of_bad_footballs}. In these lemmas, \(F,F_{1}\) and \(F_{2}\) are disjointly embedded bad-footballs, and we set \(A = F \cap E(\Sigma)\), $A_1= F_1\cap E(\Sigma)$ and $A_2 = F_2\cap E(\Sigma)$.

\begin{lemm}
\label{lem:FootballAnnulusNotBoundaryParallel}
The annulus $A$ is not boundary parallel in \(E(\Sigma)\), and \(\partial A\) is essential in \(\partial N(\Sigma)\).
\end{lemm}

\begin{proof}
Lemma~\ref{Lem:NoBadOrbsInNS} implies that \(A\) is not boundary parallel, and since each of its boundary components bounds a disk
of \(F \cap N(\Sigma)\) 
that intersects \(\Sigma\) once, both are essential in \(\partial N(\Sigma)\).
\end{proof}

\begin{lemm}
\label{lem:FootballAnnuliNotParallel}
If $A_1$ and $A_2$ are parallel in $E(\Sigma)$, then $F_1$ and $F_2$ are parallel in $\OO$.
\label{parallelfootball}
\end{lemm}

\begin{proof}
Suppose $A_1$ and $A_2$ co-bound the region \(A_{1} \times [0,1]\) in $E(\Sigma)$, and let \(\OO'\) be the component of \(\OO\) cut open along \(F_{1} \cup F_{2}\) containing \(A_{1} \times [0,1]\).  Then \(\partial (\OO' \cap N(\Sigma))\) consists of two bad footballs, 
and each bounds a product region  by Lemma~\ref{Lem:GoodFootballsInNS}. These product regions, together with \(A_{1} \times [0,1]\), show that \(\OO'\) is a product region between \(F_{1}\) and \(F_{2}\).
\end{proof}

The annuli discussed above are not necessarily essential, but we do have:

\begin{lemm}
\label{lem:BadFootballsYieldIncompressibleAnnuli}
Suppose that no component of \(\OO\) admits disjointly embedded teardrops of distinct weights.  Then \(F \cap E(\Sigma)\) is incompressible.  
\end{lemm}

\begin{proof}
This is clear, since compressing \(F \cap E(\Sigma)\) would yield disjointly embedded teardrops of weights \(a\) and \(b\), which are distinct as \(F\) is bad.
\end{proof}

\begin{proof}[Proof of Theorem~\ref{thm:Finiteness_of_bad_footballs}]
By Lemmas~\ref{lem:FootballAnnulusNotBoundaryParallel},~\ref{lem:FootballAnnuliNotParallel}, and~\ref{lem:BadFootballsYieldIncompressibleAnnuli} (whose conditions are satisfied by assumption) the intersection of a family of disjoint, pairwise non-parallel bad-footballs with \(E(\Sigma)\) forms a family of disjoint, pairwise non-parallel, non-boundary parallel, incompressible annuli whose number is bounded by~\cite{FreedmanFreedman}.
\end{proof}

%%%%%%%%%%%%%%%%%%%%%%%%%%%%%%%%%%%%%%%%%%%%%%%%%%%%%%%%%%%%%%%%%%%%%

\section{Local Pictures about a Bad-Football}\label{sec:Local Pictures about a Bad-Football}

Although Theorem~\ref{mainthm} requires us to deal with teardrops first, we start with bad-footballs as the discussion is simpler.  In Theorem \ref{localfootball} we will discover three families of compact 3-orbifolds that satisfy the requirements of Theorem~\ref{mainthm}. The very detailed treatment is meant to serve as a warm-up for the more complicated Theorem~\ref{thm:LocalPictureAboutTeardrop}.

\begin{thm}[Local Pictures about a Bad-Football] 
\label{localfootball}
Let \(\OO\) be a closed, orientable 3-orbifold and \(F \subset \OO\) an embedded bad-football.  Then there exists a compact 3-suborbifold $X$ with $F \subset X \subset \OO$ satisfying the following properties (here, \(B^{3}\) is the interior of an embedded smooth closed ball): 
\begin{enumerate}
\item the underlying topological space of $X$ is either:
	\begin{enumerate}
	\item $|X| \cong S^2\times I$ (corresponding to Form $1$), or
	\item $|X| \cong (S^2\times S^1)\backslash B^3$ (corresponding to Form $2$ or Smooth Form below),
	\end{enumerate}
\item $\partial X$ consists of one or two spherical 2-orbifolds, 
\item the orbifold structure is one of the forms listed in Figures~\ref{figure:football:SmoothForm}--\ref{figure:football:Form2} below.
\end{enumerate}

\begin{figure}[h]
\psfrag{t}{\scl[.7][$S^2(a,b)\times I$]}
\psfrag{3}{\scl[.7][$b$]}
\psfrag{2}{\scl[.7][$a$]}
\centerline{\includegraphics[scale=0.27]{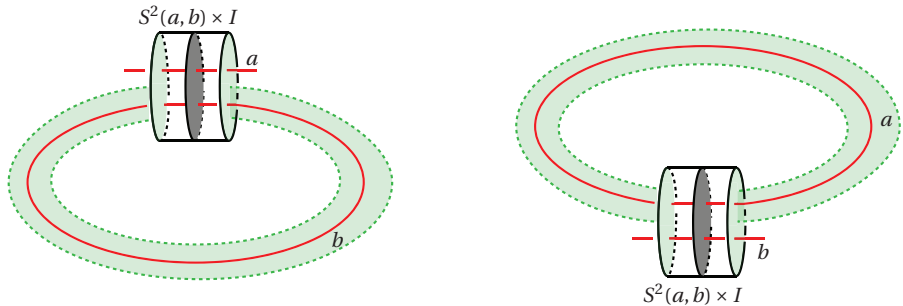}}
\caption{Smooth Form}
\label{figure:football:SmoothForm}
\end{figure}

\begin{figure}[h]
\psfrag{1}{\scl[.7][$S^2(a,a_2,a_3)$]}
\psfrag{3}{\scl[.7][$S^2(a,a_2',a_3')$]}
\psfrag{2}{\scl[.7][$S^2(a,b)$]}
\psfrag{4}{\scl[.7][$\vm$]}
\psfrag{5}{\scl[.7][$\vp$]}
\centerline{\includegraphics[scale=0.18]{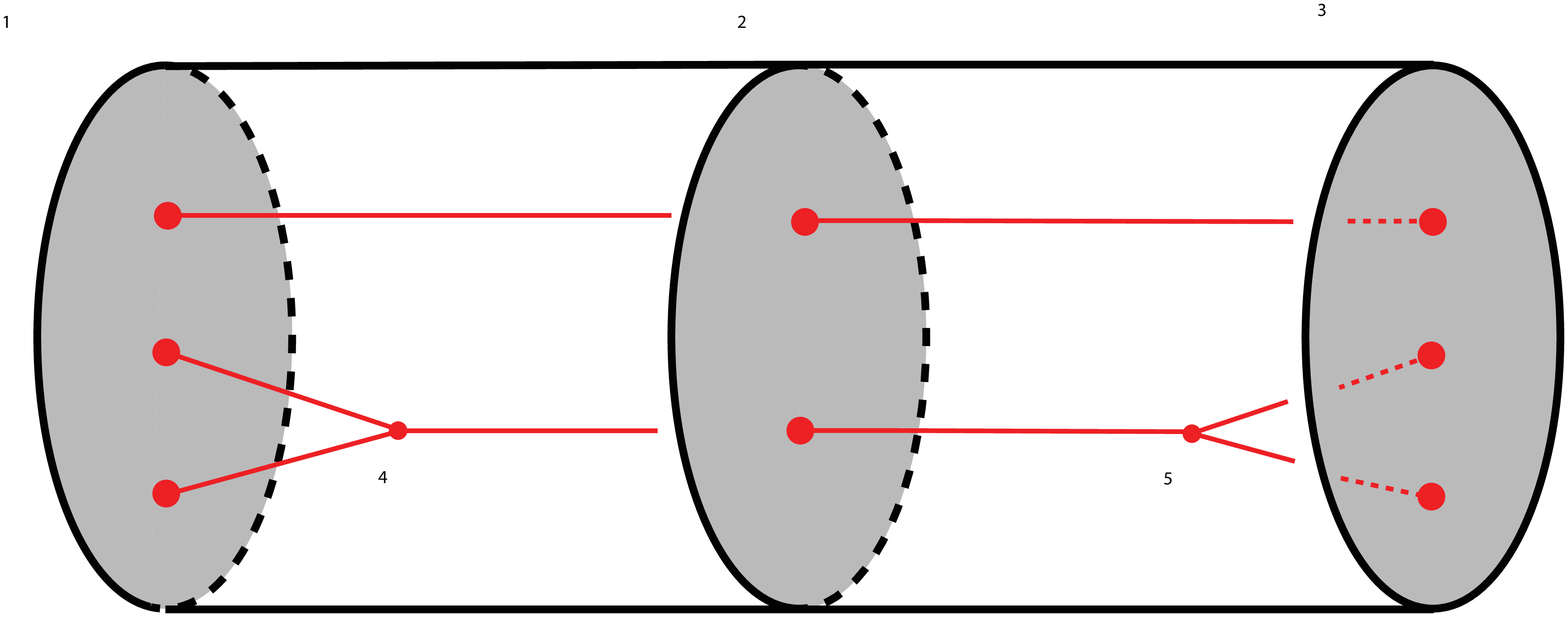}}
\caption{Form~1}
\label{figure:football:Form1}
\end{figure}

\begin{figure}[h!]
\psfrag{t}{\scl[.7][$S^2(2,3)\times I$]}
\psfrag{3}{\scl[.7][$b=3$]}
\psfrag{2}{\scl[.7][$a^{*}=2$]}
\psfrag{a}{\scl[.7][$a=2$]}
\centerline{\includegraphics[scale=0.25]{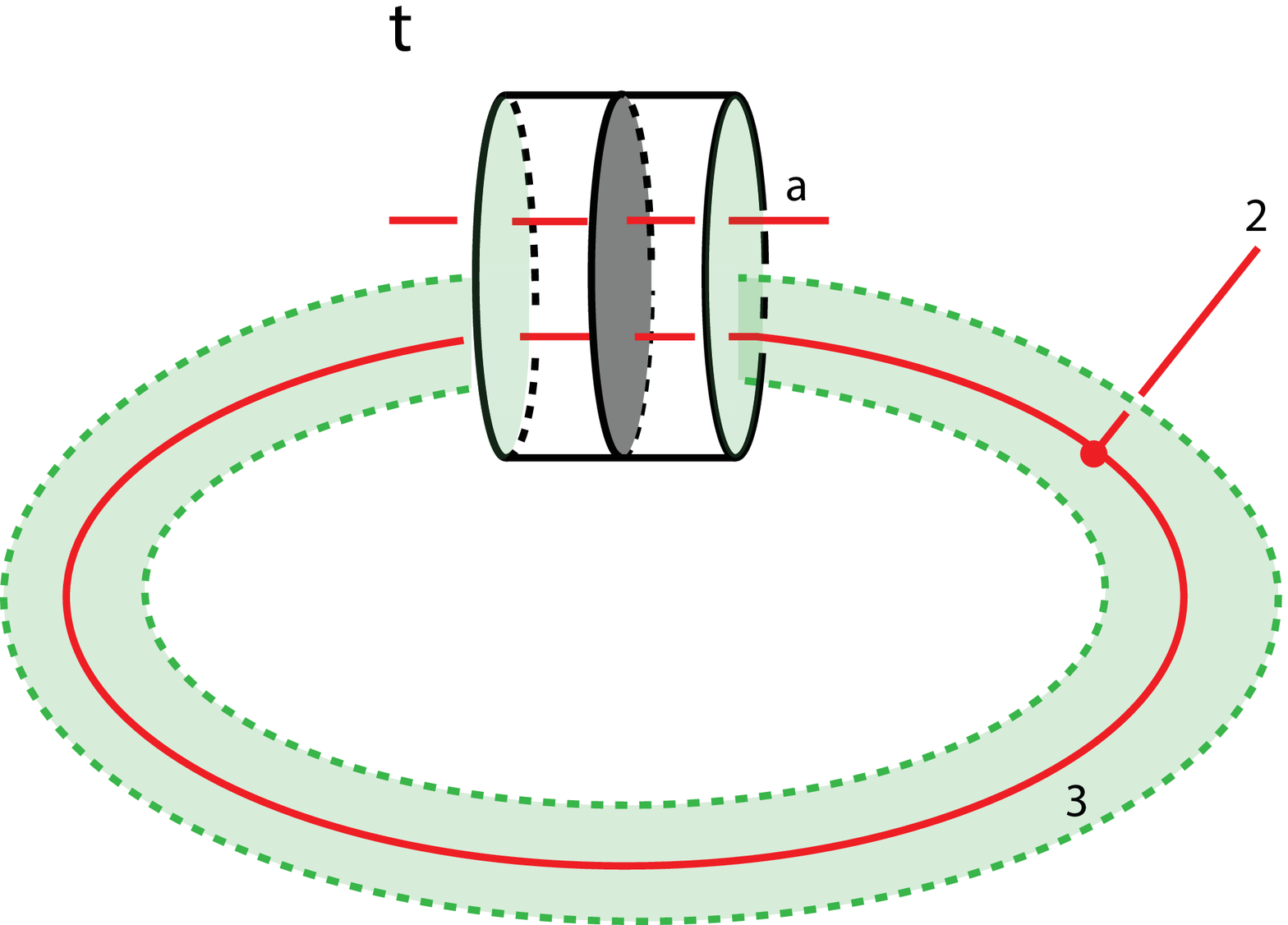}}
\caption{Form~2}
\label{figure:football:Form2}
\end{figure}
\end{thm}

\begin{proof}
\setcounter{case}{0}
A flow chart for the proof of this theorem can be found in Figure \ref{FlowchartForFootballs}. 
\begin{figure}[h!]
\psfrag{1}{\scl[.7][Yes]}
\psfrag{2}{\scl[.7][No]}
\psfrag{a}{\scl[.7][Input]}
\psfrag{b}{\scl[.7][Bad-Football]}
\psfrag{c}{\scl[.7][Is $e$ or $f$]}
\psfrag{j}{\scl[.7][smooth?]}
\psfrag{e}{\scl[.7][Smooth]}
\psfrag{f}{\scl[.7][Form]}
\psfrag{g}{\scl[.7][Form 1]}
\psfrag{h}{\scl[.7][Form 2]}
\psfrag{i}{\scl[.7][\(\vp \stackrel{?}{=} \vm\)]}
\centerline{\includegraphics[scale=0.7]{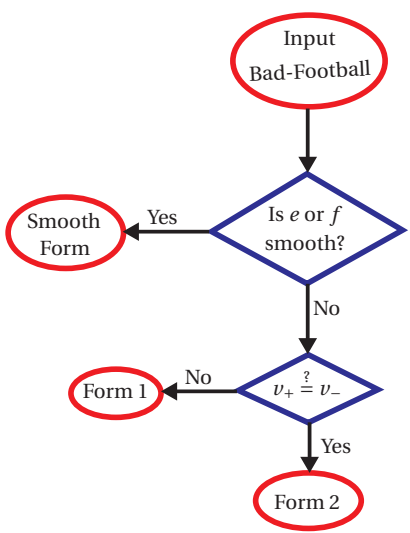}}
\caption{The flowchart diagraming the proof of Theorem \ref{localfootball}}
\label{FlowchartForFootballs}
\end{figure}
Suppose that \(F \cong S^{2}(a,b)\) with \(b>a\).
Let \(e\) and \(f\) be the edges of the singular set \(\Sigma\) that intersect \(F\) whose weights are \(b\) and \(a\) respectively (so \(e\) is heavier than \(f\)).

\begin{case}
\textbf{$\boldsymbol{e}$ or $\boldsymbol{f}$ is smooth:}
See Figure ~\ref{figure:football:SmoothForm}.
We set \(X\) to be \(N\left(S^{2}(a,b) \cup e\right)\) if \(e\) is smooth and  \(N\left(S^{2}(a,b) \cup f\right)\) if \(f\) is smooth and \(e\) is not; choosing \(e\) when both are smooth is not necessary but does make our construction of \(X\) canonical.  Clearly, $|X| \cong \left(S^2\times S^1\right)\backslash B^3$ and $\partial X=S^2(a,a)$ if $e$ is smooth and $S^2(b,b)$ otherwise.  Thus \(\partial X\) is a single good-football (which is spherical).  This is the \em Smooth Form.\em
\end{case}

\bigskip\noindent
We assume from now on that neither \(e\) nor \(f\) is smooth.  To organize the remaining cases we construct the tree \(\chi\) and the map \(\omega:\chi \to \Sigma\); this is more than we need for dealing with bad-footballs but serves to prepare the reader for Theorem~\ref{thm:LocalPictureAboutTeardrop}.  (It may be helpful to note that \(\chi\) is a subspace of the universal cover of a component of \(\Sigma\) and \(\omega\) is the restriction of the covering projection.)

\bigskip
\noindent
{\bf Construction of \(\boldsymbol{\chi}\) and \(\boldsymbol{\omega}\).} 
\label{ConstructionChiOmega}
Let \(\chi\) be the tree that consists of the six vertices and the five edges as shown in Figure~\ref{figure:chiForFootball}. Next we construct a map \(\omega:\chi \to \Sigma\); \(\omega\) is required to map the interior of every edge of \(\chi\)  homeomorphically  onto the interior of an edge of \(\Sigma\). We first map \(e'\) to \(e\) (with either orientation); this can be done since \(e\) is not smooth.  We then extend the map to \(\ep'\) and \(\oep'\), mapping them to edges adjacent to \(\omega(\vp')\) in such a way that \(\omega\) is locally injective; this can be done because the valence of the vertices of \(\Sigma\) is three. We similarly extend \(\omega\) to \(\eem'\) and \(\oem'\), obtaining \(\omega:\chi \to \Sigma\) that is locally injective although it may well fail to be injective.  We denote images under \(\omega\) by dropping the primes, for example, \(\omega(\vp') = \vp\) (thus, for example, when \(e\) forms a simple closed curve we have that \(\vp = \vm\) although  \(\vp' \not= \vm'\)).
\def\sizee{.7}
\begin{figure}[h!]
\psfrag{1}{\scl[\sizee][$\vm'$]}
\psfrag{2}{\scl[\sizee][$\vp'$]}
\psfrag{3}{\scl[\sizee][\(\oem'\)]}
\psfrag{4}{\scl[\sizee][\(\eem'\)]}
\psfrag{5}{\scl[\sizee][\(\ovmm'\)]}
\psfrag{6}{\scl[\sizee][\(\vmm'\)]}
\psfrag{7}{\scl[\sizee][\(\vpp'\)]}
\psfrag{8}{\scl[\sizee][\(\ep'\)]}
\psfrag{9}{\scl[\sizee][\(\oep'\)]}
\psfrag{a}{\scl[\sizee][\(\ovpp'\)]}
\psfrag{f}{\scl[\sizee][$e'$]}
\hspace{-1cm}
\centerline{\includegraphics[scale=0.25]{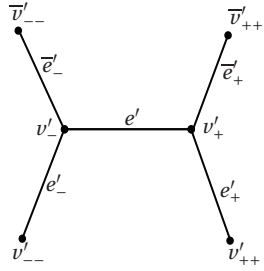}}
\caption{The tree $\chi$}
\label{figure:chiForFootball}
\end{figure}

\begin{remark}{\rm
\label{remark:embedding_in_neighborhood}
Since \(\omega\) is locally injective, whenever the restriction of \(\omega\) to a sub-tree \(K \subset \chi\) is injective, we have that \(\omega|_{N(K)}:N(K) \to \Sigma\) is injective as well, where here \(N(K)\) is a sufficiently small normal neighborhood of \(K\). This requires only local injectivity of \(\chi\) and compactness of \(K\) and will also hold for \(\chi\) and \(\omega\) constructed in the proof of Theorem~\ref{thm:LocalPictureAboutTeardrop}.}
\end{remark}

\bigskip
\noindent Having constructed \(\chi\) and \(\omega\) we now return to the proof of Theorem~\ref{localfootball} (under the assumption that \(e\) and \(f\) are not smooth).

\bigskip

\begin{case}
\textbf{\(\boldsymbol{\vp \neq \vm}\).}
See Figure~\ref{figure:football:Form1}, where \(e\) is the bottom edge.
We set $X \subset \OO$ to be  \(N\left(F \cup e \right)\). The assumption that \(\vp \neq \vm\) implies that \(\omega\) embeds \(N(e')\) into \(\Sigma\) (recall Remark~\ref{remark:embedding_in_neighborhood}), and hence \(|X| \cong S^{2} \times I\).  
Next we show that the boundary component near \(\vm\) is spherical. 
Let \(b\), $a_2$ and $a_3$ be the weights of the edges adjacent to \(\vm\) (recall that \(b\) is the weight of \(e\)).  Then the component of \(\partial X\) near \(\vm\) is isomorphic to \(S^{2}(a,a_{2},a_{3})\). Since  \(b\), $a_2$ and $a_3$ are the weights at a vertex and \(b >a\), by Lemma~\ref{firstcut} we have that \(S^{2}(a,a_{2},a_{3})\) is spherical. 
Similarly, the component of \(\partial X\) near \(\vp\) (denoted \(S^{2}(a,a_{2}',a_{3}')\) in the Figure~\ref{figure:football:Form1}) is spherical as well.
This is \em Form~1\em.
\end{case}

We have reduced the proof to the case where neither \(e\) nor \(f\) is smooth and ---

\begin{case}
\label{FootballCase3}
\textbf{\(\boldsymbol{\vp = \vm}\).}
See Figure~\ref{figure:football:Form2}, where \(e\) is the bottom edge.
The assumption \(\vp = \vm\) implies that \(e\) forms a simple closed curve that contains exactly one vertex \(v = \vp = \vm\).  We set \(X\) to be \(N\left(F \cup e\right)\).  Similar to the Smooth Form, we clearly see that $|X| \cong \left(S^2\times S^1\right)\backslash B^3$.  

It remains to show that \(\partial X\) is spherical.  Let \(a^{*}\) denote the weight of the edge connected to \(e\) at \(v\).
Clearly, \(\partial X \cong S^{2}(a,a,a^{*})\).  The weights at vertex \(v\) are \(b,b,a^{*}\) and satisfy
\begin{equation}
\label{equation:TheWeightsAtV}
\frac{1}{b} + \frac{1}{b} + \frac{1}{a^{*}} > 1
\end{equation}
This forces \(b<4\); moreover, as \(b>a>1\), we have that \(b=3\) and \(a=2\).  Equation~\eqref{equation:TheWeightsAtV} (together with \(b=3\)) forces \(a^{*} = 2\), and we conclude that
\[
\partial X \cong
S^{2}(a,a,a^{*}) \cong 
S^{2}(2,2,2)
\]
which is indeed spherical as required.  This is \em Form~2.\em
\end{case}
This completes the proof of Theorem~\ref{localfootball}.
\end{proof}

\begin{remark}{\rm
\label{rmk:MoreThanOneCase}
We could have set things up differently. For example, the assumption in Case~1 could have been \em \(e\) is smooth\em, since smoothness of \(f\) was not used in cases~(2) and~(3). We chose the specific structure (as reflected in the flowchart in Figure~\ref{FlowchartForFootballs}) to get a \em canonical \em construction and a streamlined proof, but other (equally canonical) choices are available.
}\end{remark}

%%%%%%%%%%%%%%%%%%%%%%%%%%%%%%%%%%%%%%%%%%%%%%%%%%%%%%%%%%%%%%%%%%%%%%%%%%%%%

\section{Local Pictures about a Teardrop}\label{sec:Local Pictures about a Teardrop}
In Theorem~\ref{localteardrop} we will show that every embedded teardrop in a closed, orientable 3-orbifold is contained in a compact 3-suborbifold satisfying the requirement of Theorem~\ref{mainthm}.  Before we delve into the intricacies of this theorem and its proof let us see the construction in an explicit example: 

\begin{figure}[h!]
\psfrag{s}{\scl[.7][$T \cong S^2\left(3\right)$]}
\psfrag{2}{\scl[.7][$2$]}
\psfrag{3}{\scl[.7][$3$]}
\psfrag{4}{\scl[.7][$4$]}
\psfrag{r}{\scl[.7][$N\left(\text{right}\right)$]}
\psfrag{l}{\scl[.7][$N\left(\text{left}\right)$]}
\centerline{\includegraphics[scale=0.3]{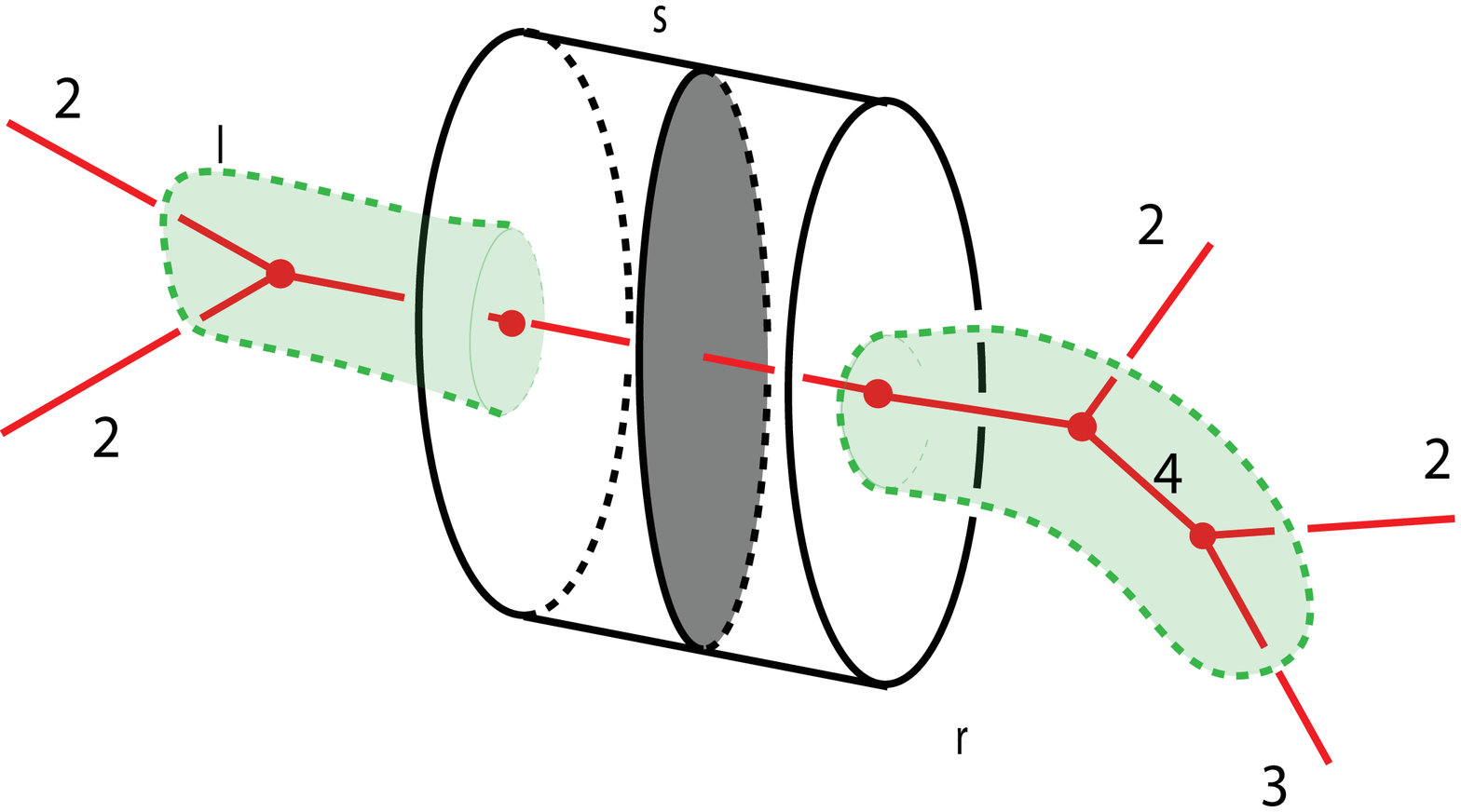}}
\caption{$X=N\left(\text{left}\right)\cup 
N\left(S^{2}(3)\right)
\cup N\left(\text{right}\right)$}
\label{teardropex}
\end{figure}
\begin{ex}
For this example see Figure~\ref{teardropex}. 
Let \(T \subset \OO\) be a teardrop in a closed orientable orbifold, \(T \cong S^{2}(3)\), and let \(e\) be the edge of \(\Sigma\) that intersects \(S^{2}(3)\) (thus the weight of \(e\) is 3).  In our example the endpoints of \(e\) are distinct, say \(\vp \neq \vm\).
We begin the construction of \(X\) with \(N\left(S^{2}(3)\right) \cong S^{2}(3) \times I\).  Next we travel to the right of $N\left(S^{2}(3)\right)$ along $e$ until we reach \(\vp\), where \(\Sigma\) branches out into two edges weights $2$ and $4$. We continue to travel along the heavier edge, of weight $4$, until we meet another vertex 
(denoted \(\vpp\) in the construction below) 
which branches out into two edges of weights $2$ and $3$.
Then we consider a normal neighborhood of the path traveled which we denote by $N(\text{right})$ (green in Figure~\ref{teardropex}). Next, we travel to the left of $N\left(S^{2}(3)\right)$ along $e$ until we reach \(\vm\), where \(\Sigma\) branches out into two edges of the same weight, $2$. Again we consider a normal neighborhood of the path traveled which we denote by $N(\text{left})$.
We set $X$ to be $N(\text{left}) \cup N\left(S^{2}(3)\right) \cup N(\text{right})$ (\(X\) is an explicit example of \em Form~2 \em from Theorem \ref{localteardrop}). Note that $|X| \cong S^2\times I$ and $\partial X$ consisits of two components, isomorphic to $S^2(2,2)$ and $S^2(2,2,3)$, which are both spherical 2-orbifolds.  

The reader may question our choices, in particular when traveling right: 
\begin{itemize}
\item If we stopped at \(\vp\) we would get  \(S^{2}(2,4)\) as a boundary component.
\item If we continued past \(\vp\) along the lighter edge (of weight \(2\)) past a vertex with weights, say, \(2,a,b\), we would get a boundary component \(S^{2}(4,a,b)\).  
\end{itemize}
This is a problem: \(S^{2}(2,4)\) is not spherical, and it is distinctly possible (depending on \(a,b\)) that \(S^{2}(4,a,b)\) is not spherical.  As discussed in the introduction, having spherical boundary components is an essential feature of Theorem~\ref{mainthm} without which we cannot possibly hope to \em cut and cap\em\ ({\it cf.} the choice of the heavier edge in Case~(\ref{FootballCase3}) of the proof of theorem~\ref{localfootball}).
\label{example}
\end{ex}

\begin{thm}[Local Pictures about a Teardrop]
\label{thm:LocalPictureAboutTeardrop}
Let $\OO$ be a closed, orientable 3-orbifold and $T \subset \OO$ an embedded teardrop. 
Then there exists a compact 3-orbifold $X$,  $T \subset X \subset \OO$, satisfying the following properties (here, \(B^{3}\) is the interior of an embedded smooth closed ball): 
\begin{enumerate}
\item for the underlying topological space $|X|$ we have either:
	\begin{enumerate}
	\item $|X| \cong S^2\times I$ (corresponding to Form $1,2$, or $3$ below), or
	\item $|X| \cong (S^2\times S^1)\backslash B^3$ (corresponding to Form $4,5,6$ or Smooth Form below),
	\end{enumerate}
\item  $\partial X$ consists of one or two spherical 2-orbifolds, and
\item the orbifold structure is one of the forms listed in Figures~\ref{figure:TearDropSmoothForm}--\ref{figure:TearDropForm6}.
\end{enumerate}

\bigskip

\def\LabelSize{.65}

\begin{figure}[h!]
\psfrag{s}{\scl[\LabelSize][$S^2(a)\times I$]}
\centerline{\includegraphics[scale=0.15]{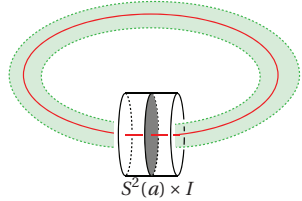}}
\caption{\noindent \cob[Smooth Form]}
\label{figure:TearDropSmoothForm}
\end{figure}

\bigskip

\begin{figure}[h!]
\psfrag{3}{{\scl[\LabelSize][$S^2(\ap,\oap)$]}}
\psfrag{1}{{\scl[\LabelSize][$S^2(\am,\oam)$]}}
\psfrag{2}{{\scl[\LabelSize][$S^2(a)$]}}
\centerline{\includegraphics[scale=0.12]{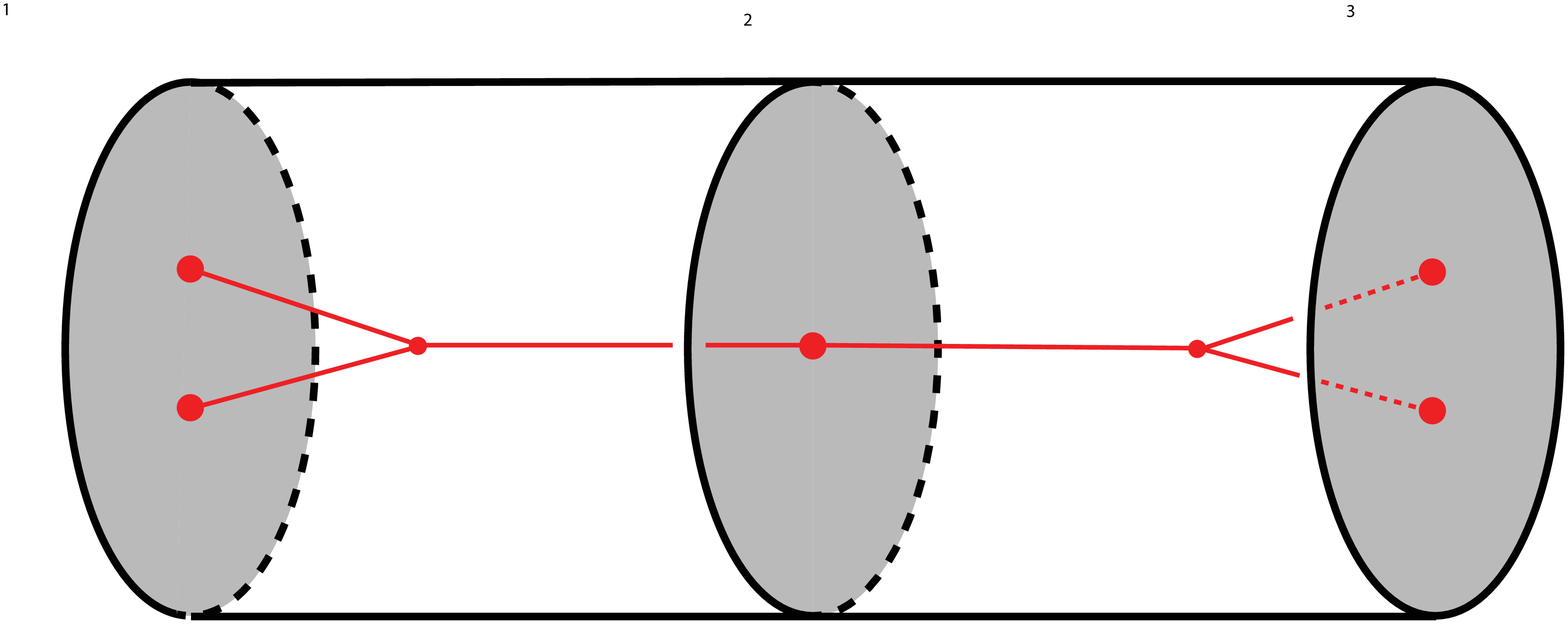}}
\caption{\cob[Form 1] (\(\ap=\oap\) and \(\am=\oam\))} 
\label{figure:form1}
\end{figure}

\bigskip

\begin{figure}[h!]
\psfrag{1}{{\scl[\LabelSize][$S^2(\ap,\oap)$]}}
\psfrag{3}{{\scl[\LabelSize][$S^2(\oam,a_{2},a_{3})$]}}
\psfrag{2}{{\scl[\LabelSize][$S^2(a)$]}}
\centerline{\includegraphics[scale=0.12]{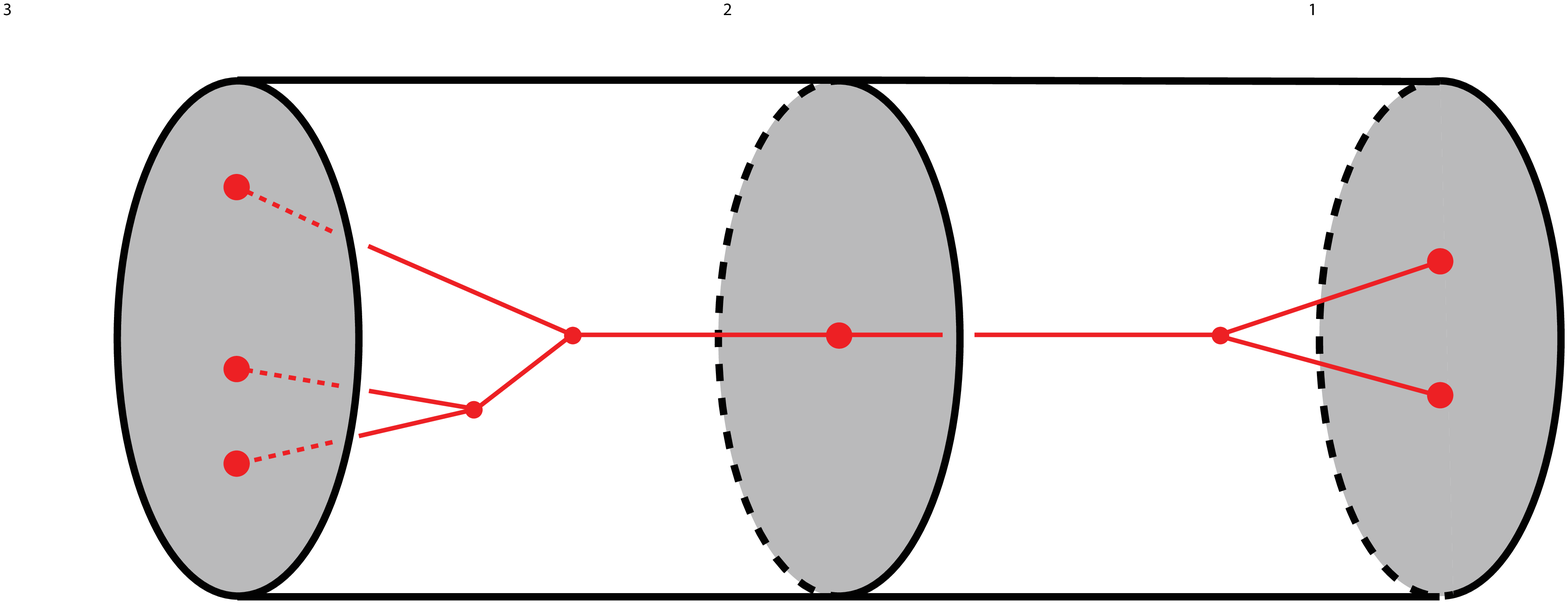}}
\caption{\cob[Form 2] (\(\ap=\oap\))}
\label{figure:form2}
\end{figure}

\bigskip

\begin{figure}[h!]
\psfrag{1}{{\scl[\LabelSize][$S^2(\oam,a_{2}',a_{3}')$]}}
\psfrag{3}{{\scl[\LabelSize][$S^2(\oap,a_{2},a_{3})$]}}
\psfrag{2}{{\scl[\LabelSize][$S^2(a)$]}}
\centerline{\includegraphics[scale=0.12]{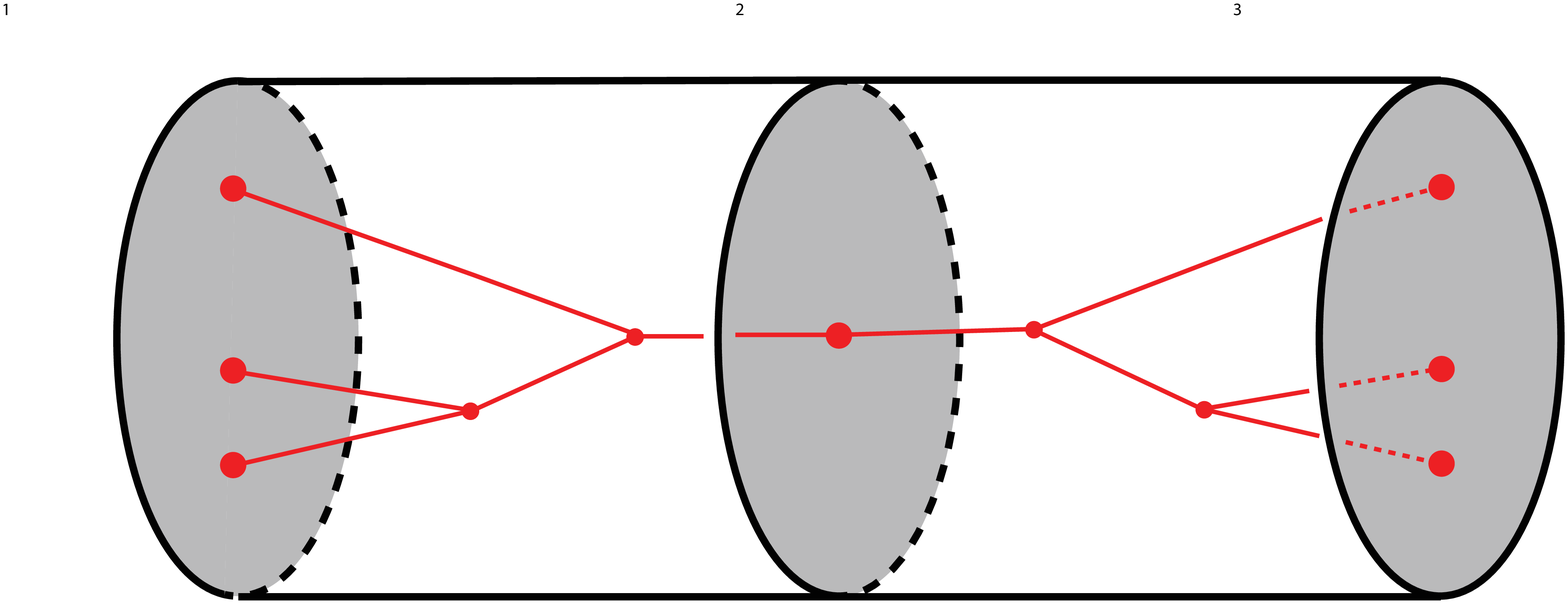}}
\caption{\cob[Form 3]}
\label{figure:form3}
\end{figure}

\begin{figure}[h!]
\psfrag{e}{\scl[\LabelSize][$S^2(a)\times I$]}
\psfrag{n}{\scl[\LabelSize][$\ap$]}
\psfrag{m}{\scl[\LabelSize][$m$]}
\psfrag{r}{\scl[\LabelSize][$\app$]}
\psfrag{o}{\scl[\LabelSize][$\oapp < \app$]}
\psfrag{O}{\scl[\LabelSize][$\oapp = \app$]}
\psfrag{V}{\scl[\LabelSize][$\vpp$]}
\psfrag{v}{\scl[\LabelSize][$v$]}
\psfrag{p}{\scl[\LabelSize][$a_1$]}
\psfrag{q}{\scl[\LabelSize][$a_2$]}
\centerline{\includegraphics[scale=0.2]{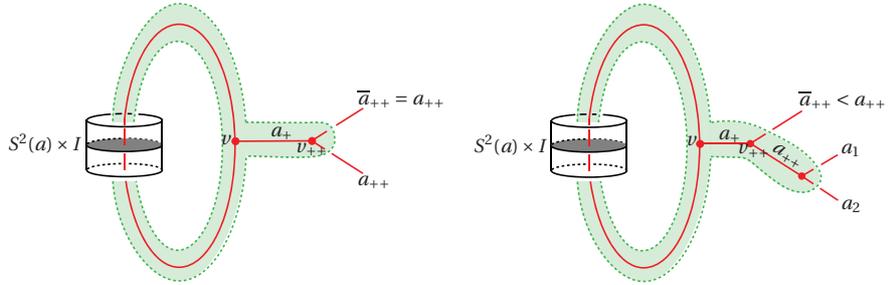}}
\caption{\cob[Forms 4.a (left) and 4.b (right)]}
\label{figure:form4}
\end{figure}

\begin{figure}[h!]
\psfrag{2}{\scl[\LabelSize][$a_{1} < a^{*}_{1}$]}
\psfrag{3}{\scl[\LabelSize][$a_{1}^{*}$]}
\psfrag{1}{\scl[\LabelSize][$\vp$]}
\psfrag{4}{\scl[\LabelSize][$\vm$]}
\psfrag{5}{\scl[\LabelSize][$e^{*}$]}
\psfrag{a}{}
\psfrag{b}{\scl[\LabelSize][$a_{2}$]}
\psfrag{c}{\scl[\LabelSize][$a_{3}$]}
\psfrag{s}{\scl[\LabelSize][$S^2(2)\times I$]}
\psfrag{m}{\scl[\LabelSize][$a_{1}$]}
\psfrag{M}{\scl[\LabelSize][$a_{1}$]}
\psfrag{n}{}
\psfrag{t}{\scl[\LabelSize][$S^2(a)\times I$]}
\psfrag{w}{}
\psfrag{v}{}
\centerline{\includegraphics[scale=0.2]{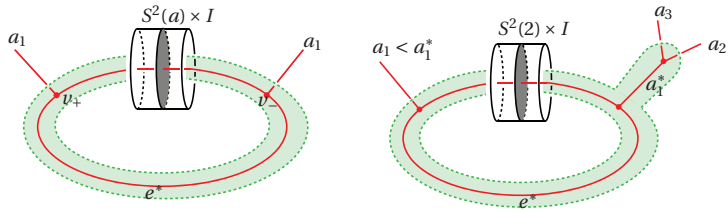}}
\caption{\cob[Forms 5.a (left) and 5.b (right)]}
\label{figure:form5}
\end{figure}

\begin{figure}[h]
\hspace{-4cm}
\psfrag{b}{\scl[\LabelSize][$\oap = 2\text{ or }3$]}
\psfrag{d}{\scl[\LabelSize][$\oam = 2$]}
\psfrag{s}{\scl[\LabelSize][$S^2(a)\times I$]}
\psfrag{a}{\scl[\LabelSize][$\ap = 3, 4 \text{ or }5$]}
\psfrag{A}{\scl[\LabelSize][$a \leq 5$]}
\psfrag{f}{\scl[\LabelSize][$\am = 3$]}
\psfrag{v}{\scl[\LabelSize][$\vp$]}
\psfrag{c}{\scl[\LabelSize][$\vmm=\vpp$]}
\psfrag{u}{\scl[\LabelSize][$\vm$]}
\psfrag{e}{\scl[\LabelSize][$a^{*} = 2$]}
\begin{center}
\centerline{\includegraphics[scale=0.2]{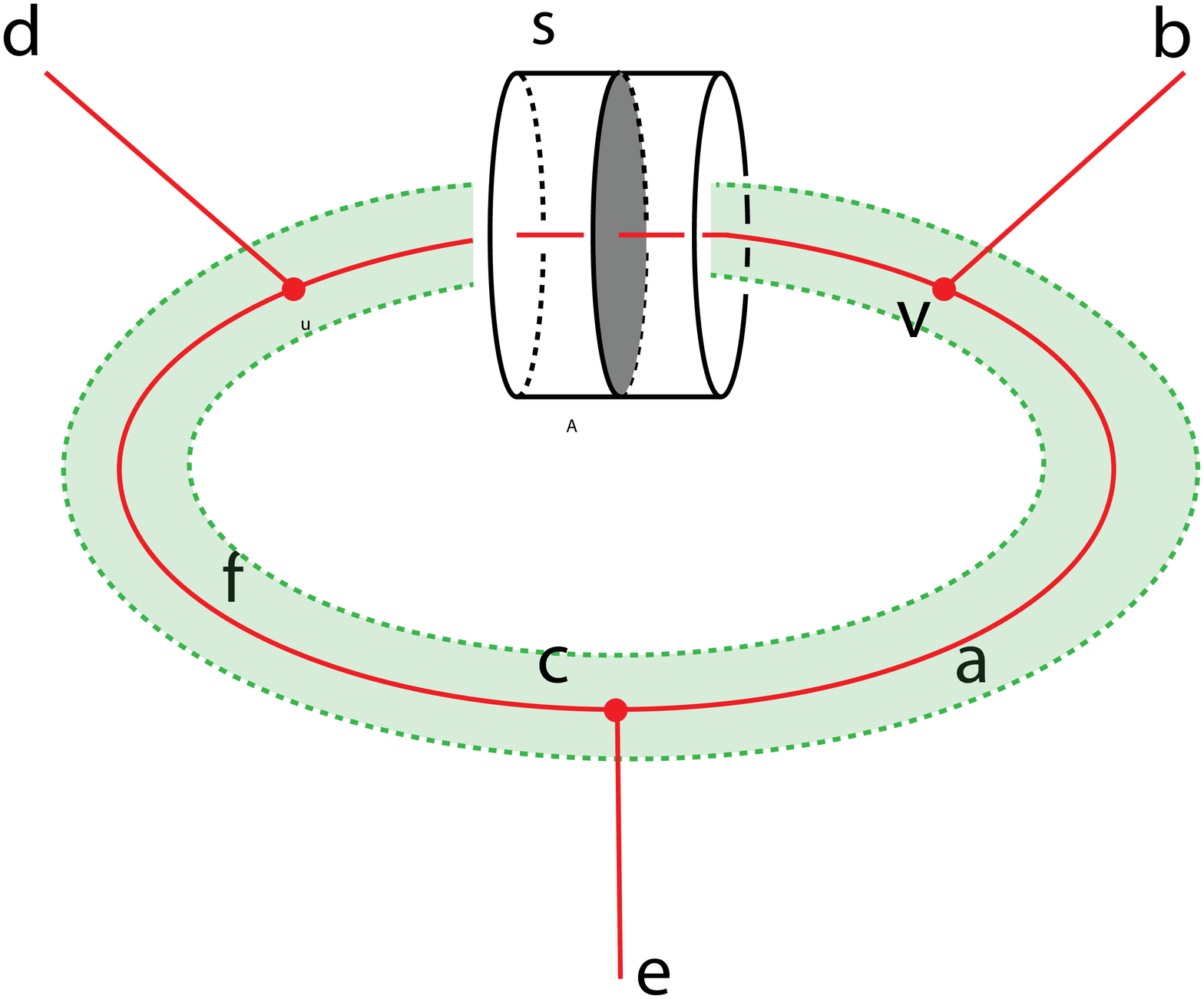}}
\end{center}
\caption{\cob[Form 6] }
\label{figure:TearDropForm6}
\end{figure}
\label{localteardrop}
\end{thm}

\begin{proof}
\setcounter{case}{0}

A flow chart for the proof of this theorem can be found in Figure~\ref{FolwChart:TearDrop}.
\begin{figure}[h!]
\psfrag{y}{\scl[.5][Yes]}
\psfrag{n}{\scl[.5][No]}
\psfrag{1}{\scl[.8][Enter Teardrop]}
\psfrag{2}{\scl[.8][Is \(e\) smooth?]}
\psfrag{3}{\scl[.8][Smooth Form]}
\psfrag{4}{\scl[.8][\(\vp \stackrel{?}{=} \vm\)]}
\psfrag{5}{\scl[.8][Form~4]}
\psfrag{6}{\scl[.8][\(N(e)\) embeds]}
\psfrag{7}{\scl[.8][\(\exists\)?]}
\psfrag{8}{\scl[.8][an edge parallel]}
\psfrag{9}{\scl[.8][to \(e\)]}
\psfrag{A}{\scl[.8][Form~5]}
\psfrag{B}{\scl[.8][\(\ap \stackrel{?}{=} \oap\)]}
\psfrag{C}{\scl[.5][and]}
\psfrag{D}{\scl[.8][\(\am \stackrel{?}{=} \oam\)]}
\psfrag{E}{\scl[.8][Form~1]}
\psfrag{F}{\scl[.8][\(\ap > \oap\)]}
\psfrag{G}{\scl[.8][\(\am \stackrel{?}{=} \oam\)]}
\psfrag{H}{\scl[.8][Form~2]}
\psfrag{I}{\scl[.8][Form~3]}
\psfrag{J}{\scl[.8][Form~6]}
\psfrag{K}{\scl[.8][\(\vpp \stackrel{?}{=} \vmm\)]}
\centerline{\includegraphics[scale=0.4]{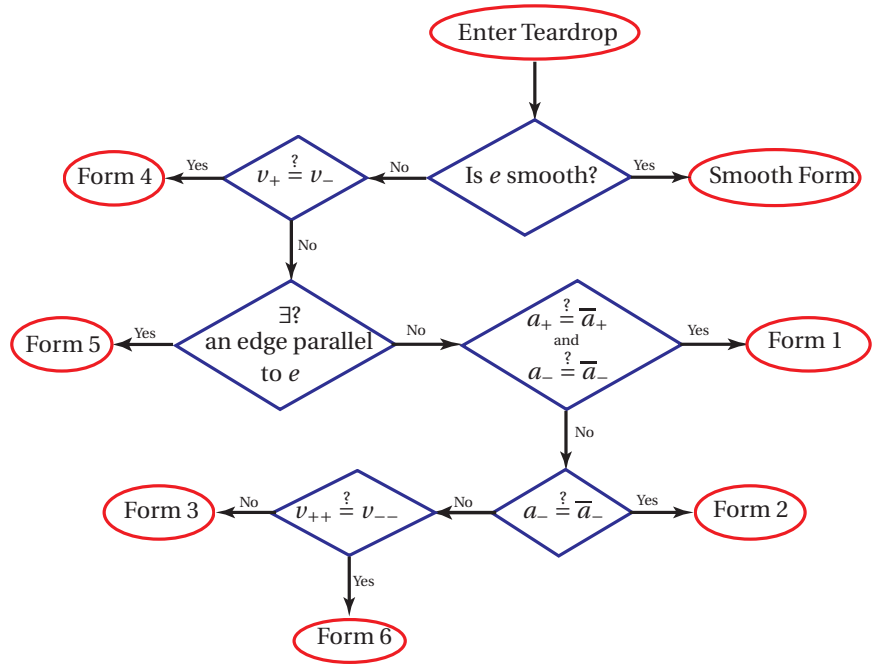}}
\caption{The flowchart diagraming the proof of Theorem \ref{localteardrop}}
\label{FolwChart:TearDrop}
\end{figure}
We note that the labels in Figures~\ref{figure:TearDropSmoothForm}--\ref{figure:TearDropForm6} correspond to weights.
Let us consider the teardrop $T$ and the edge\footnote{For the way we use the term \em edge \em recall Definition~\ref{Definition:TypesOfEdges}.}
\(e\) of $\Sigma$ that intersects it.  

\begin{case}
\textbf{$\boldsymbol{e}$ is smooth.} See Figure~\ref{figure:TearDropSmoothForm}.
When $e$ is a smooth simple closed curve we set \(X\) to be \(N\left(S^{2}(a) \cup e \right)\).  Clearly $X \cong \left(S^2(a)\times S^1\right) \setminus B^3$, where $a>1$. Therefore,  $|X| \cong \left(S^2\times S^1\right) \setminus B^3$ and $\partial X=S^2$.  This is the \em Smooth Form.\em\  
\end{case}

\medskip
\noindent
For the remainder of the proof we assume that \(e\) is not smooth. Before proceeding we construct \(\chi\) and \(\omega\) ({\it cf.} the construction of \(\chi\) and \(\omega\) on Page~\pageref{ConstructionChiOmega}). Again we note that \(\chi\) is a subspace of the universal cover of a component of \(\Sigma\) and \(\omega\) is the restriction of the covering projection to \(\chi\).

\medskip

{\bf Construction of \(\boldsymbol{\chi}\) and \(\boldsymbol{\omega}\).} 
Let \(\chi\) be the tree shown in Figure~\ref{notation}. We construct a map \(\omega:\chi \to \Sigma\); \(\omega\) is required to map the interior of every edge of \(\chi\) homeomorphically onto the interior of an edge of \(\Sigma\). We first map \(e'\) to \(e\) (with either orientation); this can be done since \(e\) is not smooth.  We then extend the map to \(\ep'\) and \(\oep'\), mapping them to edges adjacent to \(\omega(\vp')\) so that \(\omega\) is locally injective; this can be done because \(\Sigma\) is tri-valent. We similarly extend \(\omega\) to the rest of \(\chi\), obtaining \(\omega:\chi \to \Sigma\) that is locally injective although it may well fail to be injective.  We denote images under \(\omega\) by dropping the primes, for example, \(\omega(e') = e\).
Since the edges of \(\Sigma\) are weighted, \(\omega\) induces weights on the edges \(e',\ep',\oep',\eem'\) and \(\oem'\) which we denote \(a,\ap,\oap,\am\) and \(\oam\).  
\def\scl{\scalebox{.75}}
\begin{figure}[h!]
\psfrag{1}{{\scl{$\oem'$}}}
\psfrag{2}{\scl{$\vm'$}}
\psfrag{3}{\scl{$\vp'$}}
\psfrag{4}{\scl{$\ovpp'$}}
\psfrag{6}{\scl{$\oep'$}}
\psfrag{7}{\scl{$e'$}}
\psfrag{8}{\scl{$\eem'$}}
\psfrag{9}{\scl{$\ep'$}}
\psfrag{a}{\scl{$\vmm'$}}
\psfrag{b}{\scl{$\vpp'$}}
\psfrag{c}{\scl{$\ovpp'$}}
\psfrag{d}{\scl{$\ovmm'$}}
\centerline{\includegraphics[scale=0.15]{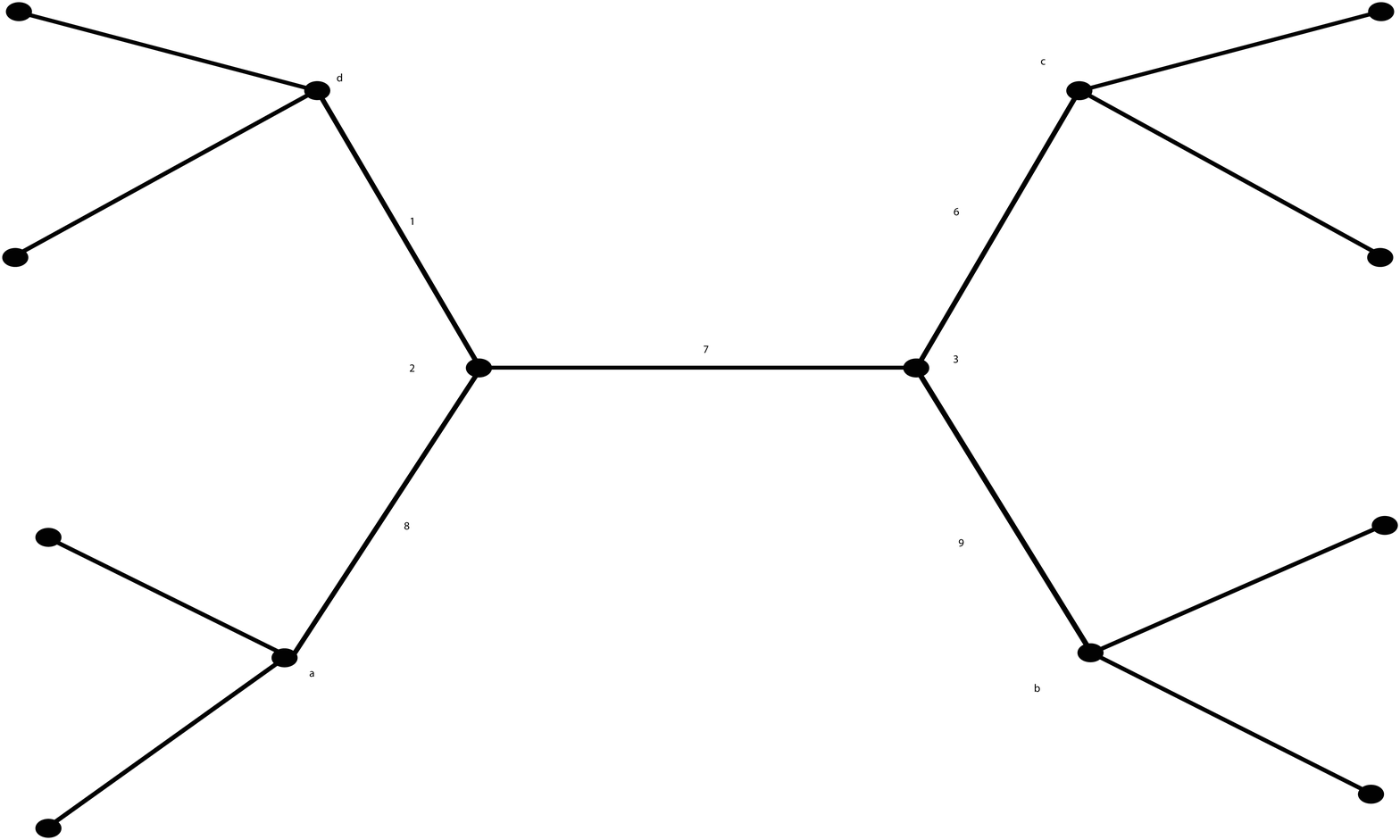}}
\caption{\(\chi\)}
\label{notation}
\end{figure}

\bigskip
\noindent
With the construction of \(\chi\) and \(\omega\) behind us we are now ready to proceed with the proof of Theorem~\ref{localteardrop}.  Recall that we already dispensed with Case~1, and we continue the work under the assumption that \(e\) is not smooth.

\noindent\begin{case}
\textbf{\(\boldsymbol{\vp = \vm}\).} See Figure~\ref{figure:form4}. We use \(v\) to denote \(\vp = \vm\). Since the valence of \(v\) is three we have that \(\ep=\oep = \eem = \oem\) and we use \(\ep\) to denote this edge and \(\ap\) to denote its weight.  Let \(\app\) and \(\oapp\) be the weights of the edges attached to \(\ep\) at \(\vpp\).  We assume as we may that \(\app \geq \oapp\) and distinguish two subcases:

\bigskip\noindent
{\bf Subcase~2.a: \(\boldsymbol{\app = \oapp}\).}  We take \(X\) to be \(N\left(T \cup e \cup \ep  \right)\) (see the left part of Figure~\ref{figure:form4}).  Since \(e\) is a simple closed curve and \(\ep \cap e\) is exactly one endpoint of \(\ep\) (since the valence of \(\vpp\) is three) we have that \(|X| \cong S^{2} \times S^{1} \setminus B^{3}\) and \(\partial X \cong S^{2}(\app,\app)\).  This is \em Form~4.a\em.

\bigskip\noindent
{\bf Subcase~2.b: \(\boldsymbol{\app > \oapp}\).}  Let \(\epp\) be the edge of weight \(\app\) that is connected to \(\ap\), and take \(X\) to be \(N\left(T \cup e \cup \ep \cup \epp \right)\) (see the right part of Figure~\ref{figure:form4}).  As above \(e\) is a simple closed curve and \(\ep\) intersects it in a single point. Moreover, \(\epp\) intersects \(e \cup \ep\) in a single point since the other endpoint of \(\epp\) cannot be \(v\) (since its valence is three) nor can it be \(\vp\), since the third edge connected to \(\vp\) is of weight \(\oapp \neq \app\).
Therefore we have that \(|X| \cong S^{2} \times S^{1} \setminus B^{3}\) and \(\partial X \cong S^{2}(\oap,a_{1},a_{2})\), where here \(a_{1}\) and \(a_{2}\) are the weights of the edges attached to \(e_{++}\).  Since \(\app > \oapp\) and \(\app,a_{1},a_{2}\) are the weights at a vertex, by Lemma~\ref{firstcut} we have that \( S^{2}(\oapp,a_{1},a_{2})\) is spherical. 
This is \em Form~4.b\em.
\end{case}
 
\bigskip\noindent We assume from now on that \(e\) is not smooth and \(\vp \neq \vm\).  By \em parallel edges \em we mean edges that connect the same vertices; they are not assumed to co-bound a disk in the orbifold.

\bigskip

\noindent\begin{case}
\label{case:ParallelEdge}
\textbf{Some edge of \(\boldsymbol{\Sigma}\) is parallel to \(\boldsymbol{e}\).}  See Figure~\ref{figure:form5}. Let \(e^{*}\) be an edge parallel to \(e\); if there are two such edges 
we choose the heavier of the two if their weights are distinct, and we choose one of the two arbitrarily otherwise.  This is not essential but does make the construction of \(X\) canonical.  Let \(a_{1} \leq a_{1}^{*}\) be the weights of the edges connected to \(e\) and \(e^{*}\) at \(\vp\) and \(\vm\). We distinguish two subcases:

\bigskip\noindent
{\bf Subcase~3.a: \(\boldsymbol{a_{1} = a_{1}^{*}}\).} 
See the left part of Figure~\ref{figure:form5}.  We take \(X\) to be \(N\left(T \cup e \cup e_{*}\right)\).  Since \(e \cup e^{*}\) is a simple closed curve we have that \(|X| \cong S^{2} \times S^{1} \setminus B^{3}\). By construction \(\partial X \cong S^{2}(a_{1},a_{1})\).  This is \em Form~5.a\em.

\bigskip\noindent
{\bf Subcase~3.b: \(\boldsymbol{a_{1} < a_{1}^{*}}\).} 
Let \(e_{1}^{*}\) be the edge of weight \(a_{1}^{*}\) (see the right part of Figure~\ref{figure:form5}). We set \(X\) to be \(T \cup e \cup e^{*} \cup e_{1}^{*}\).  Since \(a_{1} \neq a_{1}^{*}\) we have that \(e_{1}^{*} \cap \left(e \cup e^{*}\right)\) is a single vertex. 
It follows that \(|X| \cong S^{2} \times S^{1} \setminus B^{3}\).  Moreover, \(\partial X \cong S^{2}(a_{1},a_{2},a_{3})\), where here \(a_{2}\) and \(a_{3}\) are the weights of the edges connected to \(e_{1}^{*}\).   Since \(a_{1}^{*} > a_{1}\) and \(a_{1}^{*},a_{2},a_{3}\) are the weights at a vertex, by Lemma~\ref{firstcut} we have that \( S^{2}(a_{1},a_{2},a_{3})\) is spherical. This is \em Form~5.b\em.
\end{case}

\bigskip\noindent We assume from now on that \(e\) is not smooth, \(\vp \neq \vm\), and no edge of \(\Sigma\) is parallel to \(e\).  We then have the following (recall the labels as in Figure~\ref{notation}):

\begin{enumerate}
\item \(\vmm \neq \vp\) and \(\vpp \neq \vm\) (otherwise we would be in Case~\ref{case:ParallelEdge}).  
Similar statements hold  \(\ovpp\), and \(\ovmm\) but we will not need them.  
(However we cannot rule out \(\vpp = \vp\), \(\ovpp = \vp\), \(\vmm = \vm\), and \(\ovmm = \vm\).)

\item By reversing the embedding \(e' \to e\) and relabeling, \(\ep \leftrightarrow \oep\) and \(\eem \leftrightarrow \oem\), if necessary, we may assume further that one of the following holds:
	\begin{enumerate}
	\item \(\ap = \oap\) and \(\am = \oam\). 
	\item \(\ap = \oap\) and \(\am > \oam\). 
	\item \(\ap > \oap\) and \(\am > \oam\). 
	\end{enumerate}
\end{enumerate}
This leaves us with the three cases below:

\bigskip

\noindent\begin{case}
\textbf{\(\boldsymbol{\ap = \oap\) and \(\am = \oam}\).} See Figure~\ref{figure:form1}. We set \(X\) to be \(N(T \cup e)\). Clearly \(|X| \cong S^{2} \times [0,1]\) and \(\partial X\) consists of the two good footballs \(S^{2}(\ap,\oap)\) and \(S^{2}(\am,\oam)\). This is \em Form~1.\em
\end{case}

\noindent\begin{case}
\textbf{\(\boldsymbol{\ap = \oap\) and \(\am > \oam}\).} See Figure~\ref{figure:form2}. 
Since \(\am > \oam\) we have that  \(\eem\) and \(\oem\) are distinct edges, that is, \(\vmm \neq \vm\).  Furthermore, we know that \(\vmm \neq \vp\) (assumption~(1) above).
Therefore \(e \cup \eem\) is homeomorphic to a closed interval.
We set \(X\) to be \(N(T \cup e \cup e_{+})\). Clearly \(|X| \cong S^{2} \times [0,1]\) and \(\partial X\) consists of two components, one of which is the good football \(S^{2}(\ap,\oap)\) and the other \(S^{2}(\oam,a_{2},a_{3})\), where here \(a_{2}\) and \(a_{3}\) are the weights of the edges attached to \(\eem\) at \(\vmm\).  
Since \(\am,a_{2},a_{3}\) are the weights at a vertex and \(\oam < \am\), by Lemma~\ref{firstcut} we have that \(S^{2}(\oam,a_{2},a_{3})\) is spherical.  This is \em Form~2.\em
\end{case}

\noindent\begin{case}
\textbf{\(\boldsymbol{\ap > \oap}\) and \(\boldsymbol{\am > \oam}\).} We distinguish two subcases:

\bigskip\noindent
{\bf Subcase~6.a: \(\boldsymbol{\vmm \neq \vpp}\).} See Figure~\ref{figure:form3}.  The assumption that \(\vmm \neq \vpp\) together with the assumptions above imply that \(e_{-} \cup e \cup e_{+}\) homeomorphic to a closed interval. We set \(X\) to be \(N(T \cup e \cup e_{+} \cup e_{-})\). Clearly \(|X| \cong S^{2} \times [0,1]\).  Morevoer, \(\partial X\) consists of two components,  \(S^{2}(\oap,a_{2},a_{3})\) and  \(S^{2}(\oam,a'_{2},a'_{3}\)), where here \(a_{2}\) and \(a_{3}\) (respectively, \(a_{2}'\) and \(a_{3}'\)) are the weights of the edges attached to \(\ep\) at \(\vpp\) (respectively, to \(\eem\) at \(\vmm\)).  
Similar to Case~5, by Lemma~\ref{firstcut} both components of \(\partial X\) are spherical.  This is \em Form~3.\em

\bigskip\noindent
{\bf Subcase~6.b: \(\boldsymbol{\vmm = \vpp}\).} See Figure~\ref{figure:TearDropForm6}. We again take \(X\) to be \(N(T \cup e \cup e_{+} \cup e_{-})\). This time, however, the assumption that \(\vmm = \vpp\) implies that \( e \cup e_{+} \cup e_{-}\) forms a simple closed curve. We therefore see that \(|X| \cong (S^{2} \times S^{1}) \setminus B^{3}\) and \(\partial X \cong S^{2}(\oap,\oam,a^{*})\), where  \(a^{*}\) is the weight of the edge attached to \(\ep\) and \(\eem\) at \(\vpp=\vmm\).

It remains to show that \(\partial X\) is spherical.  
Since all the weights are integers no less than \(2\), 
the assumptions of Case~6 imply that \(\ap \geq 3\) and \(\am \geq 3\).  
The weights at \(\vpp\) are \(\ap, \am,\) and \(a^{*}\), and satisfy (as the weights about any vertex do) 
\[
\frac{1}{\ap} + \frac{1}{\am} + \frac{1}{a^{*}} \geq 1
\]
This forces (recall the list of spherical 2-orbifolds on Page~\pageref{2orbifolds}) at least one of the weights to be \(2\), showing that \(a^{*} = 2\).  
Furthermore, at least one of \(\ap \leq 3\) or \(\am \leq 3\) holds, and we may assume that \(\am = 3\). 
Since \(3=\am > \oam \geq 2\), we have that \(\oam = 2\).   Thus 
\[
\partial X \cong S^{2}(\oap,\oam,a^{*}) \cong S^{2}(\oap,2,2)
\]
Thus \(\partial X\) is spherical (regardless of \(\oap\)).  (The keen reader will have observed that the labels in Figure~\ref{figure:TearDropForm6} provide more information; we will leave it as an exercise to justify this, as it is not used in the proof.) This is \em Form~6.\em
\end{case}
\end{proof}

\section{Proof of Main Theorem}\label{sec:Proof of Main Theorem}

Before proving Theorem~\ref{mainthm} we state and prove a few preliminary lemmas.

\begin{lemm}
Let $Y$ be a 3-orbifold with spherical boundary and let $\hat{Y}$ be the 3-orbifold obtained from \(Y\) by capping its boundary. Then any collection of disjointly embedded 2-orbifolds in $\hat{Y}$ can be isotoped into $Y$.
\label{lemma:IsotopingIntoXi}
\end{lemm}

\begin{proof}
Each cap is the cone on a boundary component.
Since embedded 2-orbifolds cannot intersect the vertices of $\Sigma$, any such collection can be isotoped out of the caps and into \(Y\).
\end{proof}

This is the setup for Lemma~\ref{lemma:FewerTeardrops}: let $\OO_{i-1}$ be a closed, orientable orbifold admitting an embedded teardrop (the index is chosen to coincide with indices below).  Let $T\subset \OO_{i-1}$ be  a member of a maximal\footnote{Here by \em maximal \em we mean a family with the largest possible number of teardrops (similarly below, footballs).}  
 family of disjointly embedded, non-parallel teardrops (a maximal family exists by Theorem~\ref{badbounded}). Let \(X_{i} \subset \OO_{i-1}\) be the sub-orbifold containing \(T\) guaranteed to exist by Theorem~\ref{localteardrop}.  Finally, let \(\OO_{i}\) be the orbifold obtained from \(\OO_{i-1}\) by cutting-and-capping \(X_{i}\).

\begin{lemm}
The maximal number of disjointly embedded, non-parallel teardrops in $\OO_i$ is strictly less than the maximal number of disjointly embedded, non-parallel teardrops in $\OO_{i-1}$.
\label{lemma:FewerTeardrops}
\end{lemm}

\begin{proof}
Let $t_{\text{max}}$ be the maximal number of disjointly embedded, non-parallel teardrops in $\OO_{i-1}$. Suppose towards contradiction that $\mathcal{T}_{i}\subset\OO_{i}$ is a collection of $t_{\text{max}}$ disjointly embedded, non-parallel teardrops. Since \(X_{i}\) is obtained by capping \(\OO_{i-1} \setminus X_{i}\), by Lemma~\ref{lemma:IsotopingIntoXi} we may isotope \(\mathcal{T}_{i}\) out of the caps and consider $\mathcal{T}_{i}\subset \OO_{i-1}\backslash X_i$.  As $T\subset X_i$, we have that 
$$\mathcal{T}_{i}\cap T=\emptyset.$$ 
Therefore, $\mathcal{T}_{i}\cup T$ form $t_{\text{max}}+1$ disjointly embedded teardrops in $\OO_{i-1}$ and by definition of \(t_{\text{max}}\) one of the following holds:
\begin{enumerate}
\item There exists $T',T''\in\mathcal{T}_{i}$ that are parallel in $\OO_{i-1}$: let $P\subset\OO_{i-1}$ be a product region that $T'$ and $T''$ co-bound. Since $T'\cap X_{i}, T''\cap X_{i}=\emptyset$, one of the following two cases holds:
\begin{enumerate}
\item $X_i\subset P$: By considering all of the forms of $X_i$ from Theorem \ref{localteardrop}, we see that $\Sigma\cap X_i$ contains either a smooth simple closed curve (Smooth Form) or a vertex (all other forms). Therefore $X_{i}\not\subset P$, as a product region contains no simple closed curve or vertex of the singular set.

\item $X_i\cap P=\emptyset$: Then $P\subset \OO_{i}$ showing that $T'$ and $T''$ are parallel in $\OO_{i}$, contrary to our assumption.
\end{enumerate}
Therefore, case (1) cannot happen. 

\item $T$ is parallel to $T'$, for some $T'\in\mathcal{T}_{i}$: Let \(e\) be the edge of $\Sigma$ that meets $T$.\footnote{ 
Keep in mind that an edge may be a simple closed curve, as explained in Definition~\ref{Definition:TypesOfEdges}.}
By considering all the forms of $X_i$ in Theorem \ref{localteardrop} we know that $e\subset X_i$. As $T'$ is parallel to $T$, the orbifold point of $T'$ must also lie on $e$.  Therefore $T'$ cannot be isotoped out of \(X_{i}\), contradicting our construction. Therefore, case (2) cannot happen.
\end{enumerate}
This establishes Lemma~\ref{lemma:FewerTeardrops}.
\end{proof}

This is the set-up for Lemma~\ref{lemma:FewerBadFootballs}: let $\OO_{n+i-1}$ be a closed, orientable orbifold admitting an embedded bad-football but no embedded teardrop (the index is chosen to coincide with indices below). 
Let $F \subset \OO_{n+i-1}$ be a member of a maximal family of disjointly embedded, non-parallel bad-footballs (a maximal family exists by Theorem~\ref{thm:Finiteness_of_bad_footballs}). Let \(X_{n+i} \subset \OO_{n+i-1}\) be the sub-orbifold containing \(F\) guaranteed to exist by Theorem~\ref{localfootball}.  Finally, let \(\OO_{n+i}\) be the orbifold obtained from \(\OO_{n+i-1}\) by cutting-and-capping \(X_{n+i}\).

\begin{lemm}
The maximal number of disjointly embedded, non-parallel bad-footballs in $\OO_{n+i}$ is strictly less than the maximal number of disjointly embedded, non-parallel bad-footballs in $\OO_{n+i-1}$. Furthermore, there does not exist a teardrop in $\OO_{n+i}$.
\label{lemma:FewerBadFootballs}
\end{lemm}

\begin{proof}
Let $f_{\text{max}}$ be the maximal number of disjointly embedded, non-parallel footballs in $\OO_{n+i-1}$. Suppose towards contradiction that $\mathcal{F}_{n+i}\subset\OO_{n+i}$ is a collection of $f_{\text{max}}$ disjointly embedded, non-parallel footballs. By Lemma \ref{lemma:IsotopingIntoXi} we may assume that $\mathcal{F}_{n+i}\subset \OO_{n+i-1}\backslash X_{i}$. As $F \subset X_{n+i}$ we have that
$$\mathcal{F}_{n+i}\cap F=\emptyset.$$ 
Therefore, $\mathcal{T}_{n+i}\cup F$ form $f_{\text{max}}+1$ disjointly embedded bad-footballs in $\OO_{n+i-1}$ and by definition of \(f_{\text{max}}\) one of the following holds:
\begin{enumerate}
\item There exists $F^{'},F^{''}\in\mathcal{F}_{n+i}$ that are parallel in $\OO_{n+i-1}$: let $P\subset\OO_{i-1}$ be a product region that $F^{'}$ and $F^{''}$ co-bound. Since $F^{'}\cap X_{n+i}, F^{''}\cap X_{n+i}=\emptyset$, one of the following two cases holds:
\begin{enumerate}
\item $X_{n+i} \subset P$: By considering all of the forms of $X_{n+i}$ from 
Theorem~\ref{localfootball},
we see that $\Sigma\cap X_{n+i}$ contains either a smooth simple closed curve (Smooth Form) or a vertex (all other forms). Therefore $X_{n+i}\not\subset P$, as these do not exist in a product region.
\item $X_{n+i} \cap P=\emptyset$: Then $P\subset \OO_{n+i}$ showing that $F^{'}$ and $F^{''}$ are parallel in $\OO_{n+i}$, contrary to our assumption.
\end{enumerate}
Therefore, case (1) cannot happen. 

\item $F$ is parallel to $F'$, for some $F'\in\mathcal{T}_{n+i}$: Let \(e\) and \(f\) be the edges of $\Sigma$ that meet $F$.  By considering all the forms of $X_{n+i}$ in Theorem~\ref{localfootball} we see that $e \subset X_{n+i}$ or $f \subset X_{n+i}$. As $F'$ is parallel to $F$, at least one orbifold point of $F'$ must lie on $e$ or \(f\).  
Therefore $F'$ cannot be isotoped out of \(X_{i}\), contradicting our construction. 
Therefore, case (2) cannot happen.
\end{enumerate}
If \(\OO_{n+i}\) admitted a teardrop \(T\) then, by Lemma~\ref{lemma:IsotopingIntoXi} we could isotope \(T\) into \(\OO_{n+i-1}\), contradicting our assumptions.

This establishes Lemma~\ref{lemma:FewerBadFootballs}.
\end{proof}

\begin{proof}[Proof Theorem~\ref{mainthm}]

We construct $n$ and $\OO_{i}$ $(i=0,\dots,n)$ as follows: set $\OO_{0}=\OO$ and assume we have constructed $\OO_{i-1}$.
  \begin{enumerate}
  \item If $\OO_{i-1}$ does not contain a teardrop, set $n=i-1$.
  \item If $\OO_{i-1}$ does contain a teardrop, let \(T\) be a member of a maximal family of disjointly embedded,  non-parallel teardrops, and let $X_i$ be a suborbifold containing \(T\) guaranteed to exist by Theorem \ref{localteardrop}.  Construct \(\OO_{i}\) by  cutting-and-capping $X_{i} \subset \OO_{i-1}$.
  \end{enumerate}
By Theorem \ref{badbounded} the number of disjointly embedded, pairwise non-parallel teardrops in \(\OO_{i}\) is finite, and by Lemma~\ref{lemma:FewerTeardrops} this number is smaller than the corresponding number in \(\OO_{i-1}\).  Therefore the process terminates, yielding  \(n\) with $\OO_{n}$ not containing a teardrop.

\bigskip

\noindent We next construct $m$ and $\OO_{n+i}$ $(i=1,\dots,m)$ as follows: starting with \(\OO_{n}\) above, assume we have constructed $\OO_{n+i-1}$ not containing a teardrop.
  \begin{enumerate}
  \item If $\OO_{n+i-1}$ does not contain a bad-football, set $m=i-1$.
  \item If $\OO_{n+i-1}$ does contain a teardrop, let \(F\) be a member of a maximal family of disjointly embedded, mutually non-parallel bad-footballs, and let $X_{n+i}$ be a suborbifold containing \(F\) guaranteed to exist by Theorem~\ref{localfootball}.  Construct \(\OO_{n+i}\) by cutting-and-capping $X_{n+i} \subset \OO_{n+i-1}$.
  \end{enumerate}
By Theorem \ref{thm:Finiteness_of_bad_footballs} the number of disjointly embedded, pairwise non-parallel footballs in \(\OO_{n+i}\) is finite, and by Lemma~\ref{lemma:FewerBadFootballs} this number is smaller than the corresponding number  in \(\OO_{n+i-1}\). Lemma~\ref{lemma:FewerBadFootballs} also tells us that \(\OO_{n+i}\) does not admit a teardrop, and hence we may continue this process, eventually yielding \(m\) with $\OO_{n+m}$ not containing a bad-football or a teardrop.
\end{proof}

%\nocite{*}

%%%%%%%%%%%%%%%%%%%%%%%%%%%%%%%%%%%%%%%%%%%%%%%%%%%%%%%%%%%%%%%%%%%%%%%%%%%%%

\end{document}